\documentclass[a4paper,reqno]{amsart}
\numberwithin{equation}{section}

\def\?{(?)\marginpar{|?}}

\newtheorem{theorem}{Theorem}[section]

\newtheorem{lemma}[theorem]{Lemma}

\newtheorem{corollary}[theorem]{Corollary}
\newtheorem{definition}[theorem]{Definition}
\newtheorem{example}[theorem]{Example}
\def\beq{\begin{equation}}
\def\eeq{\end{equation}}
\def\be{\begin{equation*}}
\def\ee{\end{equation*}}

\begin{document}

\title{EHRENFEUCHT-FRA\"ISS\'E GAMES ON A CLASS OF SCATTERED LINEAR ORDERS}

\author{F. Mwesigye and J.K. Truss}
  \address{Department of Pure Mathematics,
          University of Leeds,
          Leeds LS2 9JT, UK}
  \email{pmtjkt@leeds.ac.uk, feresiano@yahoo.com}

\keywords{scattered linear order, monomial, Ehrenfeucht-Fra\"iss\'e game\\
supported by a scheme 5 grant from the London Mathematical Society}


\begin{abstract}
Two structures $A$ and $B$ are $n$-equivalent if player II has a winning strategy in the $n$-move Ehrenfeucht-Fra\"iss\'e game on $A$ and $B$. In earlier papers we studied $n$-equivalence classes of 
ordinals and coloured ordinals. In this paper we similarly treat a class of scattered order-types, focussing on monomials and sums of monomials in $\omega$ and its reverse $\omega^*$.
\end{abstract}

\maketitle

\section{\textbf{Introduction}}

In \cite{mwesigye2} we studied the equivalence of finite coloured linear orders up to level $n$ in an Ehrenfeucht-Fra\"iss\'e game, written as $\equiv_n$, which means that player II has a winning strategy 
in this game, as well as making some remarks about the infinite case. We gave some bounds for the minimal representatives in the finite case, and the infinite case for up to 2 moves. These results were 
extended in \cite{mwesigye3} to all coloured ordinals, in the monochromatic case giving a precise list of optimal representatives, and in the coloured case giving bounds, and in \cite{mwesigye4} some of the 
bounds for the finite case were improved for 3 moves, and additional details about the classification given for the 2-move case.

In this paper we tackle a class of linear orders which need not be well-ordered, where things are considerably more complicated. This is a subclass of the so-called `scattered' linear orders, being those 
which do not embed the order-type of the rational numbers. According to Hausdorff's characterization, these may be built up from 0 and 1 by means of sums over ordinals or reverse ordinals; see 
Theorem \ref{2.2}. A finer analysis of the class of scattered linear orderings in terms of finite sequences of finite ordinal-labelled trees, is given in \cite {montalban} (concentrating on characterizing 
scattered orders up to mutual embeddability---`equimorphism'). Even this class is too wide for us to analyze at this stage, and we restrict attention to `monomials', which are products of (possibly 
infinitely many) copies of $\omega$ or its reverse $\omega^*$, and certain sums of these. First we briefly recall the required definitions.

For a linear order $(A, <)$ we just write $A$ provided that the ordering is clear. In the $n$-move Ehrenfeucht-Fra\"iss\'e game $G_n(A, B)$ on linear orders $A$ and $B$ (or indeed any relational structures) 
players I and II play alternately, I moving first. On each move I picks an element of either structure (his choice does not have to be from the same structure on every move), and II responds by 
choosing an element of the other structure. After $n$ moves, I and II between them have chosen elements $a_1, a_2, \ldots, a_n$ of $A$, and $b_1, b_2, \ldots, b_n$ of $B$, and player II {\em wins} if 
the map taking $a_i$ to $b_i$ for each $i$ is an isomorphism of induced substructures (that is, it preserves the ordering), and player I wins otherwise. It is not required that the $a_i$ are all distinct, 
or that the $b_i$ are all distinct (though if $a_i = a_j$ but $b_i \neq b_j$ then player II will automatically lose). There is no advantage to player I in repeating a move he has played earlier, but we 
do need to consider this option in view of `2-phase' games which arise later, where moves which are distinct on points may coincide on `blocks'. As we wish to consider varying values of $n$, we may also 
write $G(A, B)$ for the family of games $\{G_n(A, B): n \ge 1\}$.

Intuitively, I is trying to demonstrate that there is some difference between the structures, while player II is trying to show that they are at least reasonably similar. We say that $A$ and $B$ are 
{\em $n$-equivalent} and write $A \equiv_n B$, if II has a winning strategy. It is easy to see that $\equiv_n$ is an equivalence relation, and it is standard that for any $n$, there are only finitely many 
$n$-equivalence classes, so it is natural to enquire what their optimal representatives may be. The problem for general orderings seems to be quite hard, but with special conditions on the type of ordering, 
or the number of moves, some results can be obtained. For ordinals, the notion of `optimality' makes sense since we may just choose the least representative, but in the orders we examine in 
this paper, this is not clear. 

It is easy to see that structures $A$ and $B$ are elementarily equivalent, written $A \equiv B$, if and only if $A \equiv_n B$ for all $n$, so we may regard $n$-equivalence as a natural approximation to 
elementary equivalence. We remark that the downward L\"owenheim-Skolem Theorem reduces the problem of the classification of linear orderings up to elementary equivalence to the same problem for countable 
linear orderings, so if we are trying to characterize optimal representatives for $n$-equivalence, we may also restrict to the countable case. 

If $A$ and $B$ are linear orders, then $A + B$ stands for the concatenation of $A$ and $B$, that is, we first assume (by replacing by copies if necessary) that $A$ and $B$ are disjoint, and we place all 
members of $A$ to the left of all members of $B$. As a generalization of this, we may write $\sum\{A_i: i \in I\}$ for the concatenation of a family of linear orders $\{A_i: i \in I\}$ indexed by a linear 
ordering $I$. We write $A \cdot B$ for the anti-lexicographic product, $B$ `copies of' $A$, to accord with the customary use for ordinals (and unlike \cite{mwesigye3}, where {\em lexicographic} products are 
used). A linear ordering is said to be \emph{scattered} if the order-type of the rational numbers does not embed in it. In \cite{rosenstein} Corollary 2.1.8 it is shown that if $A$ and $B$ are scattered, 
then so are $A + B$ and $A \cdot B$. Following \cite{rosenstein}, we denote by $\zeta$ the order-type of the integers $\mathbb Z$, which can also be construed as $\omega^* + \omega$, and by $\eta$ the 
order-type of the rational numbers $\mathbb Q$. 

\vspace{.1in}

The following straightforward result will be used without explicit reference.

\begin{lemma} {\rm (i)} If $A \equiv_n B$, then $X + A + Y  \hspace{-.01in} \equiv_n \hspace{-.01in} X + B + Y$ and $X \cdot A \cdot Y \hspace{-.01in} \equiv_n \hspace{-.01in} X \cdot B \cdot Y$.

{\rm (ii)} If $A_i \equiv_n B_i$ for each $i \in I$, then $\sum\{A_i: i \in I\} \equiv_n \sum\{B_i: i \in I\}$. \label{1.1}  \end{lemma}

Every ordinal $\alpha$  can be written in the form $\alpha = \omega^\omega \cdot \alpha_1 + \alpha_2$ where $\alpha_2 < \omega^\omega$ and in terms of this representation, the following theorem of Mostowski 
and Tarski \cite{mostowski} helps in understanding ordinals up to elementary equivalence.  

\begin{lemma} Let $\alpha = \omega^\omega \cdot \alpha_1 + \alpha_2$ and $\beta = \omega^\omega \cdot \beta_1 + \beta_2$ be 
ordinals, where $\alpha_2, \beta_2 < \omega^\omega$. Then $\alpha$ is elementarily equivalent to $\beta$ if and only if 
$\alpha_2 = \beta_2$ and either $\alpha_1 = \beta_1 = 0$ or $\alpha_1, \beta_1 > 0$.
\label{1.2}   \end{lemma}

As remarked in \cite{rosenstein}, this result enables us to conclude that `the set of all ordinals less than $\omega^\omega \cdot 2$ forms a \textit{complete and irredundant} set of representatives of the 
elementary equivalence classes of well-orderings'. The following more precise result of Mostowski and Tarski is also needed (see \cite{rosenstein} for a proof):

\begin{lemma} For any $n >0$, and ordinal $\beta > 1$, {\rm (i)} $\omega^n \equiv_{2n} \omega^n \cdot \beta$,
{\rm (ii)} $\omega^n \not \equiv_{2n+1} \omega^n \cdot \beta$. \label{1.3}
\end{lemma}

From Lemma \ref{1.2} we know that if $\alpha$ and $\beta$ are distinct ordinals which are both $< \omega^\omega$, then $\alpha \not \equiv \beta$. Later in the paper we require the following more precise 
information about the level at which this elementary inequivalence is shown. Note that the estimate of the number of moves required is given in terms of the Cantor normal form for $\alpha$, even though
$\beta$ may have a longer Cantor normal form (but this does not matter).

\begin{lemma} If $\alpha = \omega^{m_0} + \omega^{m_1} + \ldots + \omega^{m_{k-1}}$ where $m_0 \ge m_1 \ge \ldots \ge m_{k-1}$, and $\beta < \alpha$, then $\alpha \not \equiv_{2m_0 + k} \beta$. 
\label{1.4}   \end{lemma}

\begin{proof} We use induction on $k$. If $k = 1$ then $\alpha = \omega^{m_0}$. If $\beta = 0$, player I wins in one move by playing any point of $\alpha$. Otherwise, there is $b_1 \in \beta$ such that
$\beta^{\ge b_1} \cong \omega^l$ for some $l$ (possibly 0), which player I plays on his first move. Since $\beta < \alpha$, $l < m_0$. Let $a_1 \in \alpha$ be II's reply. Then 
$\alpha^{\ge a_1} \cong \omega^{m_0}$ and $\beta^{\ge b_1} \cong \omega^l$, so player I can win in at most $2l+1 \le 2m_0 - 1$ more moves by Lemma \ref{1.3}(ii), if $l > 0$ (if $l = 0$, then $b_1$ is 
greatest in $\beta$, so I wins in 1 more move by playing any point $a_2 > a_1$). 

For the induction step, on his first move, player I plays the first point $a_1$ of $\omega^{m_1}$, and we let $b_1$ be II's reply. If $b_1$ is the least point of $\beta$, then I wins on the next move by
playing the least point of $\alpha$. Otherwise, let $b < b_1$ be such that for some $l$, $[b, b_1) \cong \omega^l$. Since $\beta < \alpha$, $l \le m_0$. If $l < m_0$, player I plays $b_2 = b$, and whatever
$a_2 < a_1$ II plays, $[a_2, a_1) \cong \omega^{m_0}$ and $[b_2, b_1) \cong \omega^l$, so I wins in $2l+1 \le 2m_0 - 1$ more moves by Lemma \ref{1.3}(ii), or by the above argument if $l = 0$. Otherwise, 
$l = m_0$ and hence if $\alpha' = \omega^{m_1} + \ldots + \omega^{m_{k-1}} \cong \alpha^{\ge a_1}$ and $\beta' = \beta^{\ge b_1}$, $\beta' < \alpha'$. By the induction hypothesis, player I can now win on 
the right in at most $2m_1 + k-1$ more moves, and since $m_1 \le m_0$, this makes at most $2m_0 + k$ in all.   \end{proof}

\vspace{.1in}

The next lemma tells us how reversal works with respect to sums and products, where for any $A$, $A^*$ stands for the linear ordering obtained by reversing the ordering on $A$.

\begin{lemma} \label{1.5}
If $A$, $B$, and $I$ are linear orderings, then 
\begin{itemize}
\item[(i)] $(A + B)^* = B^* + A^*$,
\item[(ii)] $(A \cdot B)^* = A^* \cdot B^*$,
\item[(iii)] $(\sum\{A_i: i \in I\})^\ast = \sum\{(A_i)^\ast: i\in I^\ast\}$ where $A_i$ are linear orderings.
\end{itemize}
\end{lemma}
 
In \cite{mwesigye3}, a classification is given of all ordinals up to $n$-equivalence, which is based on their Cantor normal form representation. In this paper we give some results aimed at classifying all 
scattered linear orderings up to elementary equivalence and $n$-equivalence. Since this forms a large and complicated class, we restrict attention to ones which can be reasonably easily constructed. We 
shall see (in Theorem \ref{2.3}, Corollary \ref{3.7} and Theorem \ref{4.5}, and using the downward L\"oweinheim-Skolem theorem) that there are exactly $2^{\aleph_0}$ elementary equivalence classes of 
scattered orders. We principally focus on `monomials', and sums of these indexed by a finite set, or $\omega$, $\omega^*$, or $\zeta$. 

\begin{definition} \label{1.6} A \textbf{monomial} is a linear ordering which is the restricted anti-lexicographic product of a non-empty well-ordered sequence of orderings each of which is $\omega$, 
$\omega^*$, $\zeta$, or a non-zero natural number. It is said to be a \textbf{countable monomial} if it is the product of a countable such sequence.  \end{definition}

Notice that we require the family to be well-ordered, and this is so that when we order it anti-lexicographically, it is scattered, see Lemma \ref{2.4} below. (We only need consider anti-lexicographic 
products since lexicographic ones are isomorphic to the anti-lexicographic products obtained by reversing the order of the terms; the order the terms are taken in, and whether the product is lexicographic or 
anti-lexicographic, crucially affects its behaviour.) In addition, in order to obtain a linear order at all, the product needs to be taken to be `restricted', which means that it consists of the members of 
the product having `finite support'. The general notion is this: let I and $X_i$ for all $i \in I$ be linearly ordered sets and assume that $0 \in X_i$ for all $i \in I$ (where 0 may actually lie in 
$X_i$, or is some arbitrarily chosen `default value'). The infinite restricted product $\prod_{i \in I}X_i$ then stands for the set of functions $f: I \to \bigcup_{i \in I}X_i$ such that  $f(i) \in X_i$ for 
all $i$, and $\{i: f(i) \neq 0\}$ is finite. We say that $f$ has \textit{finite support}. The order is taken to be \textit{anti-lexicographic}, that is, by last difference (which extends the definition for 
the product of two linear orders to the general case). This means that for $f, g \in \prod_{i \in I}X_i$, $f < g$ if $(\exists i)(f(i) < g(i) \wedge ((\forall j > i)f(j) = g(j)))$. The hypothesis of 
finiteness of the support guarantees that such $i$ always exists. For example, $\omega \cdot \omega \cdot \omega \cdots$ is a scattered linear ordering (just equal to the usual ordinal power 
$\omega^{\omega}$) when ordered anti-lexicographically but $ \cdots \omega \cdot \omega \cdot \omega$ is dense as is seen in Lemma \ref{2.4}.

The paper is organized as follows. In section 2 we give further preliminary results needed, leading up to the `Two-phase Lemma' \ref{2.9}. In section 3 we give our main results concerning $n$-equivalence 
of monomials. Theorems \ref{3.4} and \ref{3.5} characterize the optimal length of equivalence of monomials which are finite products of $\omega$ and $\omega^*$. This is extended to monomials involving 
infinite terms, or infinitely many terms, or both, or involving a finite term or $\zeta$. Finally in section 4 we consider sums of finitely many, $\omega$, $\omega^*$, or $\zeta$ monomials. Because of the 
complications involved, we consider arbitrary sums just for powers of $\omega$ or $\omega^*$, and sums of just {\em two} more general monomials. These results should be enough to illustrate the issues 
arising.

\section{\textbf{Preliminary results}}

Scattered linear orderings may be analyzed in terms of \textit{Hausdorff rank} as given in \cite{rosenstein} (where it was called \textit{$F$-rank}). Here we just treat the countable case.

\begin{definition}  \label{2.1}
Let $V= \bigcup \{V_{\alpha}: \alpha < \omega_1\}$ where
\begin{itemize}
\item[(i)]  $ V_0 = \{0, 1\}$,
\item[(ii)] if $\alpha > 0$, $V_\alpha$ is the set of all linear orderings which may be written in the form 
$\sum\{L_i: i \in I\}$ where $I \cong \omega, \omega^*, \zeta = \omega^* + \omega$, or finite $n$, and each $L_i$ lies in $\bigcup\{V_\beta: \beta < \alpha\}$,
\item[(iii)] the {\em Hausdorff rank} of $L$ is the least $\alpha$, if any, such that $L \in V_\alpha$. 
 \end{itemize}
\end{definition}
The following theorem characterizes all countable scattered linear orderings.

\begin{theorem}[Hausdorff,\cite{rosenstein}]  \label{2.2} A countable linear ordering $L$ is in $V$ if and only if it is a 
countable scattered linear ordering.   \end{theorem} 

Thus Lemma \ref{1.2} tells us that every ordinal is elementarily equivalent to an ordinal with Hausdorff rank at most $\omega + 1$.

In terms of Hausdorff rank, we may view $\omega$ and $\omega^*$ as the simplest (infinite) scattered orderings, and ones built up from these and finite orders by $\omega$- (or $\omega^*$-) sums as the next 
most complicated. Even here, we have $2^{\aleph_0}$ elementarily inequivalent examples, as is shown in the following well-known result. We give two other examples of such families later in the paper, in 
Corollary \ref{3.7} for (infinite) monomials, and in Theorem \ref{4.5} for $\omega$-sums of powers of $\omega$ or $\omega^*$. 

\begin{theorem} There are $2^{\aleph_0}$ distinct equivalence classes of countable scattered linear orderings modulo $\equiv$.
\label{2.3} \end{theorem}
\begin{proof} For each $X \subseteq \mathbb N$ we shall find $L(X)$ so that if $X \ne Y$, then $L(X) \not \equiv L(Y)$. We let $M_n(X)$ be of order-type $\omega^\ast + \omega + n + 2$ if $n \in X$ and of 
order-type $\omega^\ast + \omega$ if $n \not \in X$; and let $L(X)= \sum\{M_n(X): n < \omega\}$, which is obtained by concatenating $M_0(X), M_1(X), M_2(X), \ldots$ in that order. Let 
$X, Y \subseteq \mathbb N$ be such that $X \neq Y$. Let $n \in X \setminus Y$ (if $X \setminus Y \neq \emptyset$, or $n \in Y \setminus X$ otherwise). Then player I has a winning strategy in the 
$(n + 4)$-move game on $L(X)$ and $L(Y)$. For player I chooses in order the final $n+2$ elements of $M_n(X)$ for his first $n + 2$ turns. He then considers player II's $n + 2$ responses 
$b_1, b_2, \ldots ,b_{n+2}$ in $L(Y)$. If, for some $i$, there is an element of $L(Y)$ between $b_i$ and $b_{i+1}$ then player I chooses that element on his $(n+3)$rd turn and player II loses 
immediately. Otherwise, since there are no maximal discrete sequences of exactly $n+2$ elements in $L(Y)$, either $b_{n+2}$ has an immediate successor, or $b_1$ has an immediate predecessor in $L(Y)$. 
Player I chooses such an element $b_{n+3}$; suppose without loss of generality, that $b_{n+2} < b_{n+3}$. Then player II must choose an element $a_{n+3}$ of $L(X)$ beyond the last point $a_{n+2}$ of 
$M_n(X)$. But then player I chooses an element between $a_{n+2}$ and $a_{n+3}$, and as there is no element between $b_{n+2}$ and $b_{n+3}$, player II loses. This therefore provides a winning strategy 
for I. Thus if $X$ and $Y$ are distinct subsets of $\mathbb N$, then $L(X) \not \equiv L(Y)$, so there are $2^{\aleph_0}$ distinct equivalence classes of countable scattered linear orderings modulo 
$\equiv$.         \end{proof}

\vspace{.1in}

The next result shows why we insist that products are taken over well-ordered families (under the anti-lexicographic order).

\begin{lemma} \label{2.4} For any non-well-ordered set I and linear orders $\alpha_i > 1$ for all $i \in I$, $\prod_{i \in I} \alpha_i$ is not scattered.  \end{lemma}
\begin{proof} Let $\beta = \prod_{i \in I} \alpha_i$. We look at $\beta$ as a linear ordering of functions as defined above. Since $I$ is not well-ordered, it has an infinite subset $I'$ ordered in type 
$\omega^*$ and we show that the set $X$ of elements of $\beta$ whose support is contained in $I'$ is densely ordered. Let $f, g \in X$ be such that $f < g$. Let $t \in I$ be such that for $i > t$, 
$f(i) = g(i)$ and $f(t) < g(t)$. Since $f(t) \ne g(t)$, $t \in I'$ so there is $s < t $ in $I'$ outside the union of the supports of $f$ and $g$.

In the first case there is $a > 0$ in $\alpha_s$, in which case we let 
 \begin{displaymath}
                   h(i) = \left\{ \begin{array}{ll} 
                  f(i) & \textrm{if $i \ne s$}\\
                   a & \textrm{if $i = s$}\\
                   \end{array} \right.
\end{displaymath}

Then the greatest (only) point at which $f$ and $h$ differ is $s$, and since $f(s)=0<a=h(s)$, we have $f<h$. Since $t$ is the greatest point at which $h$ and $g$ differ and $h(t) = f(t) < g(t)$,  also 
$h < g$. If 0 is the greatest in $\alpha_s$ (for example if $\alpha_s = \omega^\ast$) then pick $a < 0$ in $\alpha_s$ and instead let 

\begin{displaymath}
                   h(i) = \left\{ \begin{array}{ll} 
                  g(i) & \textrm{if $i \ne s$}\\
                   a & \textrm{if $i = s$}\\
                   \end{array} \right.                       
\end{displaymath}
and once more, $f < h < g$.
\end{proof}

\vspace{.1in}

Although $\zeta$ is allowed as a term in the definition of `monomial', we now show that its role is insignificant.

\begin{lemma} \label{2.5} (i) For any $B \neq \emptyset$, $\zeta \cdot B \equiv \zeta$ and $\omega \equiv \omega + \zeta \cdot B$. 

(ii) For any orderings $A$ and $B$ for which $A$ has a last element, $\omega \cdot A + \zeta \cdot B \equiv \omega \cdot A$. \end{lemma}  

\begin{proof} (i) Take any $n \ge 1$ and we show that $\zeta \cdot B \equiv_n \zeta$. Since $\zeta$ is infinite `in both directions', player II may play in such a way that if $a_1 < a_2 < \ldots < a_r$ and 
$b_1 < b_2 < \ldots < b_r$ are the first $r$ moves by the two players in $\zeta \cdot B$ and $\zeta$, so that for each $j$, $a_j$ and $b_j$ are the moves played on the same move (not necessarily the $j$th), 
then for $1 \le j < r$, $|(a_j, a_{j+1})|$ and $|(b_j, b_{j+1})|$ are either equal, or are both $\ge 2^{n-r} - 1$. (Of course, $(-\infty, a_1)$, $(-\infty, b_1)$ , $(a_r, \infty)$, $(b_r, \infty)$ are all 
infinite.) It is easy for player II to play so that this statement is true (since the required minimum length of each interval is essentially halved at each stage), and when $r = n-1$, there is still a 
point available for him to play. 

For $\omega$ and $\omega + \zeta \cdot B$, in $n$ moves, II follows a similar strategy, in addition ensuring that $|(-\infty, a_1)|$ and $|(-\infty, b_1)|$ are either equal, or both $\ge 2^{n-r} - 1$.

(ii) Write $X = \omega \cdot A + \zeta \cdot B$, $Y = \omega \cdot A$, and $A = C + 1$. Then $X = \omega(C + 1) + \zeta \cdot B$ and $Y = \omega(C + 1)$. Thus 
$X = \omega \cdot C + \omega + \zeta \cdot B$ and $Y = \omega \cdot C + \omega$, and the result follows from (i) (and Lemma \ref{1.1}). 
\end{proof}

\vspace{.1in}

We note that part (i) applies even if $B$ is not scattered, though in this paper we restrict to the scattered case. We also note that (ii) may be false without the assumption that $A$ has a last element. 
For instance, $\omega^2 + \zeta \not \equiv_4 \omega^2$. For on his first move, I plays $a_1$ in $\zeta$. Let $b_1 \in \omega^2$ be II's reply. Now I plays the first point $b_2$ of the copy of $\omega$ 
greater than $b_1$, and II's reply $a_2$ must be in $\zeta$. Now I plays the predecessor $a_3$ of $a_2$ in $\zeta$; let $b_3 \in [b_1, b_2)$ be II's reply. Since $b_2$ is a limit ordinal, I may now play in
$(b_3, b_2)$, and II cannot respond. 

\vspace{.1in}

In any monomial, we group together like terms, so that we can alternatively define a monomial to be a (restricted, anti-lexicographic) well-ordered product of orders of the form $\omega^\alpha$, 
$(\omega^*)^\beta$, $\zeta^\gamma$, $m^\delta$ for $m \in \omega$, $m \ge 2$, and non-zero ordinals $\alpha$, $\beta$, $\gamma$, $\delta$, no two consecutive ones of the same type, and these are then 
referred to as the {\em terms} of the monomial. Any monomial of interest to us will fall into one of three categories, namely a \textit{finite monomial product}, an \textit{$\omega$-monomial product} or a 
\textit{transfinite monomial product} (where the last one is indexed by an ordinal greater than $\omega$). We refer to all finite ordinals as \textit{trivial} monomials. For example, 
$\omega^2 \cdot \omega^{\ast^2} \cdot \omega^3 \cdot \zeta \cdot 4$ is a finite monomial product and 
$(\omega^\ast)^2 \cdot \omega^5 \cdot \zeta \cdot (\omega^\ast)^2 \cdot \omega^5 \cdot \zeta \cdots \hspace{0.1in} \omega^2 \cdot \zeta^3 \cdots$ is a transfinite monomial product. We refer to 
$\alpha, \beta$ and $\gamma$ as the \emph{exponents} of the terms $\omega^\alpha, \omega^\beta$, and $\zeta^\gamma$ respectively. A monomial in which $\zeta$ appears is called a {\em $\zeta$-product}.

\begin{corollary} \label{2.6} Any $\zeta$-product is elementarily equivalent to one of the form $A \cdot \zeta$, where $\zeta$ does not occur in $A$. \end{corollary}

\begin{proof} Let $A$ be a monomial in which $\zeta$ occurs. By considering the first occurrence of $\zeta$, we may write $A = A_1 \cdot \zeta \cdot A_2$, where $\zeta$ does not occur in $A_1$. By
Lemma \ref{2.5}, $\zeta \cdot A_2 \equiv \zeta$, so $A \equiv A_1 \cdot \zeta$. \end{proof}

\vspace{.1in}

Now there are various standard elementary equivalences which are well known or easily established, such as $\zeta \equiv \zeta \cdot B$ given in Lemma \ref{2.5}. Since we are focussing on optimal 
representatives for countable scattered linear orders, these standard equivalences lead us to rule out certain monomials from consideration (and a full treatment would reduce a general monomial to one of 
the simplified ones in our list). We do not give a completely precise definition of {\em simple}, or {\em optimal} form. The idea is that it should at least be a canonical representative of an equivalence 
class under $\equiv$ (or some $\equiv_n$), and should be of least possible Hausdorff rank. If it exists, then it should be possible to recognize it as such, but it might be hard to show that it does 
{\em not} exist. The principal restrictions which seem to be required are now given. We let $\Sigma$ be the family of monomials $A$ satisfying the following:

\begin{enumerate}
\item[(a)] No finite ordinal appears as a term in $A$ except on the extreme right. 

(This is because, if $n \neq 0$ is finite and $A$ is any non-trivial monomial, then $n \cdot A \equiv A$, since $n \cdot \omega = \omega$, $n \cdot \omega^\ast = \omega^\ast$ and $n \cdot \zeta = \zeta$.) 

\item[(b)] If $\omega^\alpha$ or $(\omega^*)^\alpha$ is a term of $A$, then $\alpha \le \omega$ (by Lemma \ref{1.2}). 
 
\item[(c)] Any $\zeta$-product has the form $A \cdot \zeta$ where $\zeta$ does not appear in $A$. 
\end{enumerate}

The above discussion shows that our first main task is to classify the restricted monomials that involve only products of $(\omega^\ast)^{\alpha}$, $\omega^{\beta}$ and $n$ where $n$ appears as the last 
term of the monomial (we exclude $\zeta$ since it can increase rank by at most 1).

Finally in this section we give some basic lemmas about $n$-equivalence. The following result gives finer detail about the level of elementary equivalence, from which Lemma \ref{1.2} can be deduced, and is 
a slightly stronger version of the corresponding result of Mostowski and Tarski quoted in \cite{mwesigye3}.

\begin{lemma} For any integer $n > 0$ and non-empty linear orders $X$ and $Y$, such that either $X$ and $Y$ both have a least element, or neither do,

{\rm (i)} $\omega^n \cdot X \equiv_{2n} \omega^n \cdot Y$,

{\rm (ii)} $\omega^n \not \equiv_{2n+1} \omega^n \cdot X$ if $|X| > 1$.
\label{2.7}   \end{lemma}
\begin{proof} Note that we are just assuming that $X$ and $Y$ are linearly ordered, not necessarily well-ordered, so we cannot deduce the result from the one in \cite{mwesigye3} about ordinals.

We prove (i) by induction. If $n = 1$, player II may play so that I plays the least point of one of the orders if and only if II plays the least point of the other order, and then II can respond to 
any second move by I. (This shows why the hypotheses on $X$ and $Y$ are required; if $X$ has a least element and $Y$ does not, player I can win immediately by playing the first point of 
$\omega \cdot X$.) For the induction step, assuming the result for $n$, let $A = \omega^{n+1} \cdot X$, $B = \omega^{n+1} \cdot Y$, $A' = \omega^n \cdot X$, $B' = \omega^n \cdot Y$, so that 
$A = \omega \cdot A'$ and $B = \omega \cdot B'$. By induction hypothesis, $A' \equiv_{2n} B'$, so II has a winning strategy $\sigma$ here. In the $2(n+1)$-move game on $A$ and $B$, he follows $\sigma$ `on 
the copies of $\omega$', making sure that on each move he plays the member of the current copy of $\omega$ corresponding to what I played.

Consider the first point at which II is unable to play the member of $\omega$ corresponding to what I has played. This can only happen in the last two moves. Without loss of generality, suppose that 
I's move $a$ lies in $A$, and let $a^- < a < a^+$ where $a^-$, $a^+$ are the closest points of $A$ to $a$ already played on left and right respectively ($a^- = - \infty$ or $a^+ = \infty$ are allowed if 
there are none such), and let $b^-$, $b^+$ be the corresponding points of $B$. Since II cannot respond, $b^-$ and $b^+$ must lie in consecutive copies of $\omega$ (if there was another copy of $\omega$ in 
between, he could choose a suitable member of it; if $b^-$ and $b^+$ are in the same copy, then so are $a^-$ and $a^+$, hence also $a$, and now II can certainly play the point corresponding to $a$). Hence 
$(b^-, b^+) \cong \omega + k$ for some finite $k$, and $(a^-, a^+) \cong \omega + \omega \cdot Z + k$ for some linear order $Z$ (since the argument just given shows that $a^-$ and $a^+$ lie in distinct 
copies). By the basis case, $\omega \equiv_2 \omega + \omega \cdot Z$, so also $\omega + k \equiv_2 \omega + \omega \cdot Z + k$, and II can play successfully for the last 2 moves.

(ii) is also proved by induction. On his first move, I plays the least member $b_1$ of a copy of $\omega^n$ (but not the first) in $\omega^n \cdot X$. Let $a_1 \in \omega^n$ be II's response. If 
$a_1 = 0$, then I wins at once by playing to the left of his first move in $\omega^n \times X$. Otherwise he plays $a_2 < a_1$ in $\omega^n$ so that $[a_2, a_1) \cong \omega^k$ for some $k$, which must be 
$< n$ (it is possible that $k = 0$, as must happen for instance in the basis case, $n = 1$). Let II's second move in $\omega^n \times X$ be $b_2$. Then $(b_2, b_1) \cong \omega^n \times Z$ for
some non-empty $Z$. If $(a_2, a_1) = \emptyset$ (that is, $k = 0$) then I wins on the third move by playing in between $b_2$ and $b_1$ in $B$. Otherwise $k > 0$, and $(a_2, a_1) \cong \omega^k$, and by 
induction hypothesis, I can win in the remaining $2n - 1$ moves by playing between $a_2$ and $a_1$ in $A$, and between $b_2$ and $b_1$ in $B$, since $2k + 1 \le 2n - 1$.  \end{proof}

\vspace{.1in}

\begin{definition} \label{2.8} Let $A$ and $B$ be elementarily inequivalent linear orderings and $n$ be a non-zero natural number. Then $n$ is the \emph{optimal length of elementary equivalence} of $A$ and 
$B$ if $A \equiv_n B$ and $A \not \equiv_{n + 1} B$. \end{definition}

In this case, we shall refer to a strategy that enables player I to win in $n + 1$ moves as an \textit{optimal strategy} for 
player I, and a strategy that enables player II to win in $n$ moves as an \textit{optimal strategy} for player II. We 
denote the optimal length of elementary equivalence of a game on $A$ and $B$ by $l(A,B)$. We remark that $l(A,B)$ is a unique 
non-zero natural number unless $A$ and $B$ are elementarily equivalent (when we could write $l(A,B) = \infty$).
 
\begin{lemma}[2-phase lemma] \label{2.9} Let $A'$, $B'$ be non-empty linear orders, $m$ a positive integer, and $A = \omega^m \cdot A'$, $B = \omega^m \cdot B'$, such that $l(A', B') = n$. 

(i) If $A'$, $B'$ either both have a least element, or neither does, then $l(A, B) = 2m + n$.

(ii) If one of $A'$, $B'$ has a least element but the other does not, then $l(A, B) = 1$.
\end{lemma}
\begin{proof} Let $\sigma$ and $\tau$ be optimal strategies for players I and II in the $(n+1)$-move and $n$-move games respectively on $A'$ and $B'$.

We first give a winning strategy for player I in the $(2m + n + 1)$-move game on $A$ and $B$. For the first $n + 1$ moves, I plays according to $\sigma$ on the `copies' of $\omega^m$ in $A$, on each move 
playing the first point of the copy. Let's assume that player II plays so that he hasn't yet lost after these $n + 1$ moves. This means that the partial map determined by these $n + 1$ moves of the two 
players in $A$ and $B$ is order-preserving. Since $\sigma$ is a winning strategy, the map cannot be (strictly) order-preserving on the copies, and this means that player II must have played a point of a 
copy of $\omega^m$ which contains a point that has already been played. Therefore, player II must at some point in these first $n + 1$ moves have played a point which is not the first in its copy of 
$\omega^m$. Without loss of generality, suppose that the first such point $b$ played by II lies in $B$, and this is in response to $a \in A$ played by I. Then for some $b' < b$, and $k < m$, 
$[b', b) \cong \omega^k$. Player I now plays $b'$ as his next move. Let $a' \in A$ be II's reply. Then $a' < a$ and the interval $(a', a)$ has order-type of the form $\omega^m \cdot Z$ for some non-empty 
$Z$. If there is no point between $b'$ and $b$ (so that $k = 0$), I wins on the next move by playing in $A$ between $a'$ and $a$. Otherwise, $(b', b) \cong \omega^k$ and $(a', a) \cong \omega^m \cdot Z$, so 
by Lemma \ref{2.7}(ii), player I can win in the next $2k + 1$ moves, and $n + 2 + 2k + 1 \le n + 2 + 2m - 1 = 2m + n + 1$.

Next we note that if $A'$ has a least element and $B'$ does not then the same applies to $A$ and $B$, and I can win the game on $A'$ and $B'$, and also the game on $A$ and $B$, in 2 moves, by playing the 
least element of $A'$, or $A$ as the case may be, on his first move; since $B'$, $B$ have no least, whatever II plays on his first move, I wins by playing to its left. Thus actually 
$l(A', B') = l(A, B) = 1$, giving (ii). Similarly if $B'$ has a least but $A'$ does not. 

So now assume that $A'$, $B'$ either both have a least element, or neither does. Observe that we may assume that under $\tau$, I plays the first element of one of $A'$ and $B'$ (if it exists) if and only if 
II plays the first element of the other structure. Indeed, on all moves except the last, II {\em must} play thus, since otherwise he loses on the next move as I will play to the left of the non-first 
element. And on his last move, there can be no disadvantage for II in doing so.

We can now give a winning strategy for player II in the $(2m + n)$-move game on $A$ and $B$. For the first $n$ moves, player II uses $\tau$ on $A'$ and $B'$, playing the corresponding point within each copy 
to that played by I. There are now $2m$ moves remaining. It suffices to show that II can win the game restricted to $(a_1, a_2)$ in $A$ and $(b_1, b_2)$ in $B$, where $a_1$ and $a_2$ are consecutive points 
already played in $A$, and $b_1$ and $b_2$ are the corresponding points in $B$ (or one or other may be $\pm \infty$ if the other is the greatest or least so far played), in the remaining moves. In other 
words, we have to see that $(a_1, a_2) \equiv_{2m} (b_1, b_2)$. Note that for some ordinals $\alpha_1, \alpha_2 < \omega^m$, and $a_1', a_2' \in A'$, $b_1', b_2' \in B'$, $a_1 = (\alpha_1, a_1')$, 
$a_2 = (\alpha_2, a_2')$, $b_1 = (\alpha_1, b_1')$,  and $b_2 = (\alpha_2, b_2')$ (where these are now {\em ordered pairs}, and not intervals). Since $\tau$ is winning for II `on the blocks' in $n$ moves, $a_1' = a_2' \Leftrightarrow b_1' = b_2'$, and in this case, 
the intervals $(a_1, a_2)$ and $(b_1, b_2)$ are actually equal (to the same ordinal $\le \omega^m$), and therefore $(a_1, a_2) \equiv_{2m} (b_1, b_2)$. Otherwise, $a_1' < a_2'$ and $b_1' < b_2'$, and 
$(a_1, a_2)$ and $(b_1, b_2)$ are isomorphic to $\omega^m \times X + \alpha_2$, $\omega^m \times Y + \alpha_2$, for some non-empty $X$ and $Y$, both having a least element, and by Lemma \ref{2.7}(i), these 
are $2m$-equivalent. This analysis also applies to the case $a_1 = b_1 = -\infty$. Here by assumption on $\tau$, $a_2'$ is least in $A' \Leftrightarrow b_2'$ is least in $B'$, in which case 
$(a_1, a_2) \cong \alpha_2 \cong (b_1, b_2)$. Since $\omega^m$ has no greatest, the corresponding problem does not arise on the right.
\end{proof}

\section{\textbf{The class of monomials of countable scattered linear orderings}}

In this section we look at monomials of the forms

$M_0^{n_0} \cdot M_1^{n_1} \cdots M_{s-1}^{n_{s-1}}$,

$M_0^{n_0} \cdot M_1^{n_1} \cdot M_2^{n_2} \cdots$,

and $M_0^{n_0} \cdot M_1^{n_1} \cdots M_{\omega}^{n_{\omega}} \cdot M_{\omega+1}^{n_{\omega+1}} \cdots \hspace{0.1in} M_{\omega+t-1}^{n_{\omega+t-1}}$

\noindent (finite monomial, $\omega$-monomial, transfinite monomial respectively) where each $M_i$ is $\omega$ or $\omega^\ast$ and the two types alternate, and the $n_i$ are non-zero. Mostly we suppose 
that the two monomials we are comparing begin with the same one of $\omega$, $\omega^*$, as otherwise things are rather easy, as explained by the first result.

\begin{theorem} \label{3.1} Let $A = M_0^{n_0} \cdot M_1^{n_1} \cdot M_2^{n_2} \cdots M_{s-1}^{n_{s-1}}$ and $B = N_0^{n_0} \cdot N_1^{n_1} \cdots N_{t-1}^{n_{t-1}}$ where $s$, $t$, $m_i$, $n_j$ are positive integers, and $M_i$ alternate between $\omega$ and $\omega^*$ (or the other way round), and so do the $N_j$, and $M_0 = \omega \Leftrightarrow N_0 = \omega^*$. Then if $s = 1$ or $t = 1$ (or both), $l(A, B) = 1$, and otherwise, $l(A, B) = 2$.  \end{theorem}   
\begin{proof} For the first part, let us suppose that $s = 1$ and $M_0 = \omega$, the other cases being similar. Then $A$ has a least member and $B$ does not. Here I plays the least member $a_1$ of $A$ on 
his first move, and then to the left of whatever $b_1 \in B$ II plays, and wins on the second move.

If $s, t \ge 2$, then neither $A$ nor $B$ has a least or greatest element, and this ensures that II can win in 2 moves. To see that I can win in 3 moves, again assume for ease that $M_0 = \omega$, and I 
plays on his first move the least member $a_1$ of a copy of $\omega^{m_0}$. II's reply $b_1$ lies in a copy of $(\omega^*)^{n_0}$, so has an immediate predecessor $b_2$, which I plays on his second move. Since $a_1$ has no immediate predecessor in $A$, I wins on the third move by playing between II's moves $a_2$ and $a_1$. \end{proof}

\vspace{.1in}

The next result is of use when we consider `extra' parts of a monomial beyond the significant sections.

\begin{theorem} \label{3.2} Let $A = M_0^{n_0} \cdot M_1^{n_1} \cdot M_2^{n_2} \cdots M_{s-1}^{n_{s-1}} \cdot X$ and $B = M_0^{n_0} \cdot M_1^{n_1} \cdots M_{s-1}^{n_{s-1}} \cdot Y$ where for all $i$, 
$n_i$ is a non-zero finite number, $X$ and $Y$ are non-empty linear orders, and $M_0$, $M_1, \ldots$ alternate between $\omega$ and $\omega^*$ (or the other way round). Then

{\rm (i)}  if $X$ and $Y$ either both have least elements, or neither do, and similarly for greatest elements, then $A \equiv_{2(n_0 + n_1 + \cdots + n_{s-1})} B$,

{\rm (ii)}  if $|X| > 1$ then 

$M_0^{n_0} \cdot M_1^{n_1} \cdot M_2^{n_2} \cdots M_{s-1}^{n_{s-1}}  \not \equiv_{2(n_0 + n_1 + \cdots + n_{s-1}) + 1}
M_0^{n_0} \cdot M_1^{n_1} \cdots M_{s-1}^{n_{s-1}} \cdot X$.   \end{theorem}   
\begin{proof} (i) We use induction on $s$. If $s = 1$ then the result follows from Lemma \ref{2.7}(i) (noting that there the hypothesis was just on {\em least} elements, but here since $M_0$ is allowed to be 
$\omega$ or $\omega^*$, we need the hypothesis on least and greatest). For the induction step, we assume the result for $s$, and write 
$A' = M_1^{n_1} \cdot M_2^{n_2} \cdots M_s^{n_s} \cdot X$ and $B' = M_1^{n_1} \cdot M_2^{n_2} \cdots M_s^{n_s} \cdot Y$. By induction hypothesis, 
$A' \equiv_{2(n_1 + n_2 + \cdots + n_s)} B'$, and from Lemma \ref{2.9} we deduce that $A \equiv_{2(n_0 + n_1 + \cdots + n_s)} B$.

(ii) The basis case follows from Lemma \ref{2.7}(ii), and the induction step follows by considering $M_1^{n_1} \cdot M_2^{n_2} \cdots M_s^{n_s}$ and 
$M_1^{n_1} \cdot M_2^{n_2} \cdots M_s^{n_s} \cdot X$, which by induction hypothesis are $2(n_1 + n_2 + \cdots + n_s) + 1$-inequivalent, and again applying Lemma \ref{2.9}. \end{proof}

\vspace{.1in}

For the next result we require a slightly different strengthening of Lemma \ref{2.7}. 

\begin{lemma} \label{3.3} For any integers $n \ge m  > 0$ and non-empty linear orders $X$ and $Y$,

{\rm (i)} if neither $X$ nor $Y$ has a least element, then $\omega^m \cdot X \equiv_{2m+1} \omega^n \cdot Y$,

{\rm (ii)} if neither $X$ nor $Y$ has a greatest or least element, then 
$\omega^m \cdot X \equiv_{2m+2} \omega^n \cdot Y$.

\end{lemma}   
\begin{proof} (i) Let $A = \omega^m \cdot X$ and $B = \omega^n \cdot Y$, and write $A' = X$, $B' = \omega^{n-m} \cdot Y$. Then $A'$ and $B'$ are non-empty, and so $l(A', B') \ge 1$, and neither of them has a least element. So by Lemma \ref{2.9}(i), $l(A, B) \ge 2m + 1$.

(ii) With the same $A'$ and $B'$, this time neither $A'$ nor $B'$ has a greatest or least, so $l(A', B') \ge 2$, and by Lemma \ref{2.9}(i) again, $l(A, B) \ge 2m + 2$.
\end{proof}

\vspace{.1in}
 
We are now able to give a more general result about monomials, first comparing ones of equal lengths. 

\begin{theorem} \label{3.4} Let $A = M_0^{m_0} \cdot M_1^{m_1} \cdots M_{s-1}^{m_{s-1}} \neq B = M_0^{n_0}  \cdot M_1^{n_1}\cdots M_{s-1}^{n_{s-1}}$ where $m_i $ and $n_i$ are non-zero finite numbers for 
every $i < s$ and $M_0$, $M_1, \ldots$ alternate between $\omega$ and $\omega^*$ and let $t$ be the smallest $i$ such that $m_i \ne n_i$. (Assume without loss of generality that $m_t < n_t$.) 
\begin{itemize}
\item[(i)]   If $t = s-1$ , then $l(A,B) = {2(m_0 + m_1+ \cdots + m_t)}$,

\item[(ii)]  If $t = s-2$, then $l(A,B) = {2(m_0+ m_1 + \cdots + m_t) + 1}$, 

\item[(iii)] If $t < s-2$,  then $l(A,B)= {2(m_0 + m_1 + \cdots + m_t) + 2}$. 
\end{itemize}
\end{theorem}
\begin{proof} Throughout we assume without loss of generality that $M_0 = \omega$.

(i) We use induction. If $s = 1$ then the result follows from Lemma \ref{2.7}. Otherwise, if $s > 0$, write $A$ as $\omega^{m_0} \cdot A'$, and $B$ as $\omega^{m_0} \cdot B'$, where
$A' = M_1^{m_1} \cdot M_2^{m_2} \cdots M_{s-1}^{m_{s-1}}$ and $B' = M_1^{m_1} \cdot M_2^{m_2} \cdots M_{s-2}^{m_{s-2}} \cdot M_{s-1}^{n_{s-1}}$. By induction hypothesis, 
$l(A', B') = 2(m_1 + \ldots + m_{s-1})$. Since neither $A'$ nor $B'$ has a least element, by Lemma \ref{2.9}, $l(A,B) = 2(m_0 + m_1 + \cdots + m_{s-1})$. 

(ii) We first show that $A \not \equiv_{2(m_0+ m_1 + \cdots + m_t) + 2} B$. If $s \le 2$, then as $s-2$ exists, $s = 2$, and 
$A = \omega^{m_0} \cdot (\omega^*)^{m_1}$, $B = \omega^{n_0} \cdot (\omega^*)^{n_1}$ where $m_0 < n_0$. Thus $A = \omega^{m_0} \cdot A'$, $B = \omega^{m_0} \cdot B'$ where $A' = (\omega^*)^{m_1}$
and $B' = \omega^{n_0 - m_0} \cdot (\omega^*)^{n_1}$. Since $A'$ has a greatest and $B'$ does not, $l(A', B') = 1$, so by Lemma \ref{2.9}(i), $l(A, B) = 2m_0 + 1$ (as neither $A'$ nor $B'$ has a least). 

If $s \ge 3$, we use induction. Let $A' = M_1^{m_1} \cdot M_2^{m_2} \cdots M_{s-1}^{m_{s-1}}$ and $B' = M_1^{n_1} \cdot M_2^{n_2}\cdots M_{s-1}^{n_{s-1}}$, so that $A = \omega^{m_0} \cdot A'$ and 
$B = \omega^{m_0} \cdot B'$. By induction hypothesis, $l(A', B') = 2(m_1+ m_2 + \cdots + m_{s-2}) + 1$. Since $A'$, $B'$ do not have least elements, by Lemma \ref{2.9}(i) it follows that
$l(A, B) = 2(m_0+ m_1 + \cdots + m_t) + 1$.
 
(iii) First consider the case where $t = 0$, and we start by showing that $A \not \equiv_{2m_0 + 3} B$. Let I choose the first point $b_1$ of an $\omega^{n_0}$-block of $B$, and let $a_1$ be II's response. 
Then either $a_1$ is the first point in its $\omega^{m_0}$-block, in which case, as $M_1 = \omega^*$, there is an immediately preceding block, so I can play $a_2 < a_1$ so that 
$[a_2, a_1) \cong \omega^{m_0}$, or if not, I can play $a_2 < a_1$ so that $[a_2, a_1) \cong \omega^k$ where $k < m_0$ (possibly $k = 0$). Let $b_2$ be II's reply. Then $b_2 < b_1$ and so 
$[b_2, b_1)$ is a multiple of $\omega^{n_0}$. By Lemma \ref{2.7}(ii), $\omega^k \not \equiv_{2k+1} (b_2, b_1)$, so I can win in at most $2k + 1 \le 2m_0 + 1$ more moves (if $k = 0$, then 
$(a_2, a_1) = \emptyset$, so I wins on the next move by playing between $b_2$ and $b_1$), meaning that he has used at most $2m_0 + 3$ moves in all.

The fact that $A \equiv_{2m_0 + 2} B$ follows at once from Lemma \ref{3.3}(ii).

Now consider the general case $t > 0$, and we use induction. We write $A = \omega^{m_0} \cdot A'$ and $B = \omega^{m_0} \cdot B'$, where $A' = M_1^{m_1} \ldots M_{s-1}^{m_{s-1}}$ and 
$B' = M_1^{n_1} \ldots M_{s-1}^{n_{s-1}}$. By induction hypothesis, $l(A', B') = 2(m_1 + \ldots + m_t) + 2$, and the result now follows by appeal to Lemma \ref{2.9}.
\end{proof}

\vspace{.1in}

We extend the same result to the case of monomials of different lengths. 

\begin{theorem} \label{3.5} Let $A = M_0^{m_0} \cdot M_1^{m_1} \cdots M_{s-1}^{m_{s-1}}$ and $B = M_0^{n_0}  \cdot M_1^{n_1}\cdots M_{t-1}^{n_{t-1}}$ where $m_i $ and $n_i$ are non-zero finite numbers for 
$i < s$, $j < t$, $s < t$, and $M_0$, $M_1, \ldots$ alternate between $\omega$ and $\omega^*$ and let $u$ be the smallest $i$ such that $m_i \ne n_i$ (or, if $m_i = n_i$ for all $i < s$ then we let 
$u = s$). 
\begin{itemize}
\item[(i)]  If $u = s$ or $s-1$, then $l(A,B) = 2(m_0 + m_1+ \cdots + m_{s-2}) + 1$,

\item[(ii)] If $u = s-2$, and $m_{s-2} < n_{s-2}$, then $l(A,B)= 2(m_0 + m_1 + \cdots + m_{s-3} + m_{s-2}) + 1$, and if $n_{s-2} < m_{s-2}$, then 
$l(A,B)= 2(m_0 + m_1 + \cdots + m_{s-3} + n_{s-2}) + 2$,

\item[(iii)] If $u < s-2$,  then $l(A,B)= 2(m_0 + m_1 + \cdots + m_{u-1} + \min(m_u, n_u)) + 2$. 
\end{itemize}
\end{theorem}
\begin{proof} (i) First suppose that $s = 1$. Assume without loss of generality that $M_0 = \omega$. Thus $A = \omega^{m_0}$ and $B = \omega^{n_0} \cdot (\omega^*)^{n_1} \cdots$ where $m_0 = n_0$ if
$u = 1$ and $m_0 \neq n_0$ if $u = 0$. By Lemma \ref{2.9}(ii),  
$l(A, B) = 1 = 2(m_0 + \ldots + m_{s-2}) + 1$.

If $s = 2$, then $A = \omega^{m_0} \cdot (\omega^*)^{m_1}$ and $B = \omega^{m_0} \cdot (\omega^*)^{n_1} \cdot \omega^{n_2} \cdots$ (where again $m_1 = n_1$ or $m_1 \neq n_1$ depending on the value of $u$). 
Here if we let $A' =  (\omega^*)^{m_1}$ and $B' = (\omega^*)^{n_1} \cdot \omega^{n_2} \cdots$, then by the case $s = 1$ (with $\omega^*$ in place of $\omega$), $l(A', B') = 1$. As neither $A'$ nor $B'$ has 
a least, by Lemma \ref{2.9}(i), $l(A, B) = 2m_0 + 1$, as desired.

Now let $s > 2$ and we use induction. Writing $A' = M_1^{m_1} \cdots M_{s-1}^{m_{s-1}}$, $B' = M_1^{n_1} \cdots M_{t-1}^{n_{t-1}}$, by the induction hypothesis, $l(A', B') = 2(m_1 + \ldots + m_{s-2}) + 1$, 
and by appeal to Lemma \ref{2.9} again, $l(A, B) = 2(m_0 + m_1 + \ldots + m_{s-2}) + 1$. 

(ii) If $s = 2$ then assuming $M_0 = \omega$, $A = \omega^{m_0}(\omega^*)^{m_1}$ and $B = \omega^{n_0}(\omega^*)^{n_1}\omega^{n_2} \cdots$. If $m_0 < n_0$ then we write $A = \omega^{m_0} \cdot A'$ and
$B = \omega^{m_0} \cdot B'$, where neither $A' = (\omega^*)^{m_1}$ nor $B' = \omega^{n_0 - m_0}(\omega^*)^{n_1}\omega^{n_2} \cdots$ has a least element. Clearly $l(A', B') = 1$, so by Lemma \ref{2.9}(i), 
$l(A, B) = 2m_0 + 1$. If $n_0 < m_0$ then $A = \omega^{n_0} \cdot A'$ and $B = \omega^{n_0} \cdot B'$, where $A' = \omega^{m_0 - n_0}(\omega^*)^{m_1}$, $B' = (\omega^*)^{n_1} \omega^{n_2} \cdots$ again do not
have least elements. This time, $l(A', B') = 2$, giving $l(A, B) = 2n_0 + 2$. For I can win in 3 moves by playing the first point of the final $\omega^{m_0 - n_0}$ block in $A'$. Whatever $b_1 \in B'$ II 
plays has an immediate predecessor $b_2$ which I plays, and then I wins on the third move by playing in $A'$ between $a_2$ and $a_1$. Since however $A'$ and $B'$ have neither greatest {\em nor} least, 
$A' \equiv_2 B'$.

For general $s > 2$ we argue inductively, and write $A = \omega^{m_0}\cdot A'$, $B = \omega^{m_0} \cdot B'$. By induction hypothesis the result holds for $A'$ and $B'$, and this lifts to $A$ and $B$ by using Lemma \ref{2.9}(i).

(iii) Here we have $A = M_0^{m_0} \cdots M_{u-1}^{m_{u-1}}M_u^{m_u} \cdots M_{s-1}^{m_{s-1}}$ and $B = M_0^{m_0} \cdots M_{u-1}^{m_{u-1}}$ $M_u^{n_u} \cdots M_{t-1}^{m_{t-1}}$. First suppose that $u = 0$ and
$M_0 = \omega$, so $A = \omega^{m_0} \cdots M_{s-1}^{m_{s-1}}$, $B = \omega^{n_0} \cdots M_{t-1}^{n_{t-1}}$, $m_0 < n_0$, $s > 2$. Player I plays the first point $b_1$ of an $\omega^{n_0}$ block in $B$. 
Then as in the proof of Theorem \ref{3.4}, whatever $a_1 \in A$ II plays, I can play $a_2 < a_1$ so that $[a_2, a_1) \cong \omega^k$ for some $k \le m_0$ (in the same $\omega^{m_0}$ block as $a_1$ if it is 
not the first point of the block, or else in the immediately preceding block if it is the first). Then $[b_2, b_1)$ must be a multiple of $\omega^{n_0}$, and as $m_0 < n_0$, by Lemma \ref{2.7}(ii), I can 
win in at most $2m_0 + 1$ further moves, making at most $2m_0 + 3$ in all.

A similar argument applies if $n_0 < m_0$, and the general case follows by induction using Lemma \ref{2.9}(i) as before.
\end{proof}

\begin{corollary} \label{3.6} If $A =  M_0^{m_0} \cdot M_1^{m_1} \cdots M_{s-1}^{m_{s-1}}$ and $B =  M_0^{n_0} \cdot M_1^{n_1} \cdots M_{t-1}^{n_{t-1}}$ where $m_i, n_j$ for $i < s, j < t$
are non-zero natural numbers, then $A \equiv B$ if and only if $s = t$ and $m_i = n_i$ for all $i$.  \end{corollary}
\begin{proof} $\Longrightarrow$ It follows from Theorem \ref{3.5} that $s = t$. Suppose that $m_j \ne n_j$ for some $j$, and let $u$ be the least such. Theorem \ref{3.4} gives a precise value for $m$ that 
depends on $u$ such that $A \equiv_m B$ but $A \not \equiv_{m+1} B$.  Therefore $A \not \equiv B$ which is a contradiction. Hence $m_j = n_j$ for all $j$. 

$\Longleftarrow$  Suppose that $s = t$ and $m_i = n_i$ for all $i < s$. Then $A \cong B$ and hence $A \equiv B$. 
\end{proof}

\vspace{.1in}

Corollary \ref{3.6} tells us that the simple form of a finite product monomial, say $A =  M_0^{n_0} \cdot M_1^{n_1} \cdots M_{s-1}^{n_{s-1}}$ with finite exponents, is $A$ itself and it is unique. The 
following corollary tells us about the simple form of an $\omega$-monomial product with finite exponents. 
 
\begin{corollary} \label{3.7} If $A = M_0^{m_0} \cdot M_1^{m_1} \cdots$ and $B =  M_0^{n_0} \cdot M_1^{n_1} \cdots$ are $\omega$-monomial products where $m_i, n_i$ are non-zero natural numbers, then 
$A \equiv B$ if and only if $m_i = n_i$ for all $i$.       \end{corollary}

\begin{proof} $\Longrightarrow$ Suppose that $m_i \neq n_i$ for some $i$, and let $t$ be the smallest such $i$. By following the same strategy as in Theorem \ref{3.4}, player I can win in 
$2(n_0 + n_1 + \cdots + n_t) + 3$ moves, which shows that $A \not \equiv B$. 

$\Longleftarrow$ Conversely, if $m_i = n_i$ for all $i$ , then $A \cong B$, so $A \equiv B$.  \end{proof} 

\begin{corollary} \label{3.8} Let $A = M_0^{n_0} \cdot M_1^{n_1} \cdots$ and $B =  M_0^{n_0} \cdot M_1^{n_1} \cdots N$  where $N$ is any monomial. Then $A$ is elementarily equivalent to $B$. \end{corollary}

\begin{proof} By the proof of Theorem \ref{3.4}, for each $n$, $A \equiv_n B$, and hence $A \equiv B$. \end{proof}

\vspace{.1in}

We now consider monomials that involve infinite exponents. 

\begin{theorem} If $A = M_0^{m_0} \cdot M_1^{m_1} \cdots M_{t-1}^{m_{t-1}} \cdot M_t^{\lambda} \cdot M_{t+1}^{m_{t+1}} \cdots$ where every $m_i$ is a non-zero natural number and $\lambda \ge \omega$ and  
$B = M_0^{m_0} \cdot M_1^{m_1} \cdots M_{t-1}^{m_{t-1}} \cdot M_t^\lambda \cdot M_{t+1}^{n_{t+1}}$ where $n_{t+1} \geq 1$, then $A \equiv_n B$ for all $n$, so $A \equiv B$. \label{3.9} \end{theorem}

\begin{proof} By Theorem \ref{3.2}, for each $k$, $A \equiv_{2(m_0 + m_1 + \ldots + m_{t-1} + k)} B$, and hence $A \equiv_n B$ for all $n$.   \end{proof}

\vspace{.1in}

We remark that the same method shows that 
$M_0^{m_0} \cdot M_1^{m_1} \cdots M_{t-1}^{m_{t-1}} \cdot M_t^\lambda \cdot M_{t+1}^{m_{t+1}} \equiv M_0^{m_0} \cdot M_1^{m_1} \cdots M_{t-1}^{m_{t-1}} \cdot M_t^\mu \cdot M_{t+1}^{n_{t+1}}$. 

\vspace{.1in}

By Corollary \ref{3.8}, the simple form of a monomial which has at least one term  with infinite  exponent is obtained by replacing the first such term by the one with exponent $\omega$, replacing the next 
term if any by one with exponent 1, and removing all subsequent terms.

\begin{example} \label{3.10} {\rm (i)}  $\omega^\omega \cdot \omega ^ \ast \equiv \omega^\omega \cdot (\omega^*)^n$ for any non-zero ordinal $n$.

{\rm (ii)}  $\omega^\omega \cdot (\omega^*)^{n_1} \cdot \omega^{n_2} \equiv \omega^\omega \cdot \omega^\ast $.

{\rm (iii)}  $M_0^{n_0} \cdot M_1^{n_1} \cdots M_s^{n_s} \cdot \omega^\omega \cdot \omega^\ast  
\equiv M_0^{n_0} \cdot M_1^{n_1} \cdots M_s^{n_s} \cdot \omega^\omega \cdot (\omega^*)^{k_1} \cdot \omega^{k_2} \cdot (\omega^*)^{k_3} \cdots$.

{\rm (iv)}  $M_0^{n_0} \cdot M_1^{n_1} \cdots M_s^{n_s} \cdot \omega^\omega \cdot \omega^\ast \equiv M_0^{n_0} \cdot M_1^{n_1} \cdots M_s^{n_s} \cdot \omega^\alpha \cdot \omega^\ast $ for any 
$\alpha \ge \omega$. 
\end{example}

\begin{theorem} Any member of $\Sigma$ has a unique simple form.  \label{3.11} \end{theorem}
\begin{proof}  This is a consequence of Theorems \ref{3.4}, \ref{3.5}, \ref{3.9}, and Corollary \ref{3.8}.  \end{proof} 

\begin{theorem} Any monomial product is elementarily equivalent some monomial product of rank at most $\omega + 1$.  \label{3.12}  \end{theorem}
\begin{proof}  This also follows from \ref{3.4}, \ref{3.9}, \ref{3.9}, and \ref{3.8}, and since, as we saw earlier, $\zeta$ can only cause the rank of a monomial product to grow by at most 1 up to elementary 
equivalence and that no terms that come after $\zeta$ appear in the simple form of any given monomial.  \end{proof}

\vspace{.1in}

From the above work, we draw the following conclusions.

Any $X \in \Sigma$ is elementarily equivalent to a linear ordering of rank at most $\omega + 1$. Hence for monomials involving $\omega^\ast, \omega$, and $n$, the bound on the rank of such monomials modulo 
elementary equivalence is the same as the bound on the rank of ordinals modulo elementary equivalence. The presence of $\zeta$ causes rank to grow by at most 1 up to elementary equivalence. In fact, the 
bound on the rank of members of the new class of monomials  up to elementary equivalence is still $\omega + 1$.

Briefly considering the non-scattered case, we let $\Sigma_q$ be the monomial class of countable linear orderings which allows $\eta$, the order-type of the rationals, to appear as a term, that is, 
$\Sigma_q$ is a class of linear orderings that are products of $\omega^{r_i}, (\omega^\ast)^{s_i}, \zeta^{t_i}, \eta$ and $n$ where $s_i, k_i$ and $t_i$ are non-zero ordinals. We note that any 
$L \in \Sigma_q$ is either a scattered countable monomial (if $\eta$ does not appear), a dense countable monomial or falls in between. In fact as in Corollary \ref{2.6}, if $\eta$ occurs at all, we may 
suppose up to elementary equivalence that it occurs just once, and at the end, so that the monomial has the form $A \cdot \eta$, where $A$ is scattered.

\textbf{\section{\textbf{On sums of monomials}}}

The next stage in building scattered orderings according to Definition \ref{2.1} is to look at sums of monomials over $\omega, \omega^*$, $\mathbb Z$, or a finite set as index set. The general case is too 
complicated for our current techniques, so we just give a few illustrative special cases. We assume non-redundancy of the expression as in the following definition.

\vspace{.1in}

\begin{definition} \label{4.1} \textit{A monomial sum}, $W$, is a linear ordering which is a sum of monomials over an index set $\omega, \omega^*$, $\mathbb Z$, or $m \in \omega$, such that 
if $M$ and $N$ are consecutive terms of $W$, then $M+N \neq M, N$. \end{definition}

Each monomial occurring in a monomial sum is a \textit{term}. First we look at the cases where the terms are just powers of $\omega$ or $\omega^*$, and the index set is $\mathbb Z$. We remark on a special 
case, relevant for what follows, in which the terms on the right are eventually all powers $\omega^{n_i}$ of $\omega$, from the $k$th on, say. The definition of `monomial sum' entails that 
$n_k \ge n_{k+1} \ge \ldots$, so this sequence is eventually constant, from the $l$th term on, say. Then $\omega^{n_l} + \omega^{n_{l+1}} + \ldots = \omega^{n_l} \cdot \omega = \omega^{n_l + 1}$. The 
$\mathbb Z$-indexed sum thus {\em could} be replaced by an $\omega^*$-indexed one (though we do not do this, as we are seeking to compare similar sums of non-trivial terms). If however beyond the $k$th, all
terms are powers of $\omega^*$, the powers could increase. Dual remarks apply to behaviour on the left, with the roles of $\omega$ and $\omega^*$ interchanged.

\begin{lemma} \label{4.2} If $A = \sum \{t_i: i \in {\mathbb Z}\}$ and $B = \sum \{u_i: i \in {\mathbb Z}\}$ are monomial sums, and there are natural numbers $m_i$ and $n_i$ such that for each $i$, 
$t_i = \omega^{m_i}$ or $(\omega^*)^{m_i}$, and $u_i$ is either $\omega^{n_i}$ or $(\omega^*)^{n_i}$, then $A \cong B$ if and only if there is an automorphism $\theta$ of $\mathbb Z$ such that for all $i$, 
$u_i = t_{\theta i}$. \end{lemma}

\begin{proof} Automorphisms of $\mathbb Z$ are of course just translations. Clearly if there is an automorphism of $\mathbb Z$ as stated, then $A \cong B$. Conversely, suppose that $\varphi: A \to B$
is an isomorphism. We shall show that for each $i$, $\varphi(t_i) = u_j$ for some $j$, and the result then follows on letting $\theta(i) = j$.

First suppose that $t_i$ is a singleton, i.e. $m_i = 0$, and let $j$ be such that $\varphi(t_i) \subseteq u_j$. We show that $u_j$ is also a singleton, which establishes that $\varphi(t_i) = u_j$. If this 
does not hold, then $\varphi(t_i)$ must be a proper subset of $u_j$, so $n_j \neq 0$. Assume without loss of generality that $u_j = \omega^{n_j}$. Not all $t_k$ for $k \ge i$ can be singletons. For then they
would all have to map into $u_j$ under $\varphi$, and $u_{j+1}$ would be disjoint from the image of $\varphi$. Let $k > i$ be least such that $t_k$ is not a singleton. Then $t_k = (\omega^*)^{m_k}$ with
$m_k > 0$ (as otherwise, the singletons would be `absorbed' into $t_k$, contrary to the definition of `monomial sum'). But since $\varphi$ is an isomorphism, and the points of $B$ greater than 
$\varphi(t_i)$ begin with an $\omega$-sequence, this is impossible. 

Next assume that $m_i \neq 0$. Now $\varphi(t_i)$ intersects a finite non-empty set $J_i$ of $u_j$s. This $J_i$ is minimal such that $\varphi(t_i) \subseteq \bigcup_{j \in J_i}u_j$. Suppose without loss of 
generality that $t_i = \omega^{m_i}$, and let $j$ and $k$ be the least and greatest of $J_i$ respectively. Then each $\varphi(t_i) \cap u_r$ is well-ordered. If $j < r < k$, $u_r \subseteq \varphi(t_i)$, 
and so $u_r = \omega^{n_r}$. At the left hand end point, $I = \varphi(t_i) \cap u_j$ is a final segment of $u_j$, so either $u_j = \omega^{n_j}$, in which case $\varphi(t_i) \cap u_j \cong \omega^{n_j}$ or 
$u_j = (\omega^*)^{n_j}$ and $\varphi(t_i) \cap u_j$ is finite. At the right hand end point, $F = \varphi(t_i) \cap u_k$ is an infinite initial segment of $u_k$ isomorphic to $\omega^{m_i}$, so that 
$u_k = \omega^{n_k}$ and $m_i \le n_k$. Thus $\varphi(t_i) = I \cup \omega^{n_{j+1}} \cup \ldots \cup \omega^{n_{k-1}} \cup F \cong \omega^{n_j} + \omega^{n_{j+1}} + \ldots + \omega^{n_{k-1}} + 
\omega^{m_i}$, or $I + \omega^{n_{j+1}} + \ldots + \omega^{n_{k-1}} + \omega^{m_i}$ where $I$ is finite. In each case, this is isomorphic to $\omega^{m_i}$, so by definition of $B$ a monomial sum, either 
$j = k$, or $k = j+1$ and $I$ is finite. To sum up, there are a finite set $I$ and an initial segment $F$ of $\omega^{n_k}$ isomorphic to $\omega^{m_i}$ such that $\varphi(t_i) = I \cup F$. Applying the 
same argument to $\varphi^{-1}u_k$ (noting that $u_k$ cannot be a singleton) we see that actually $m_i = n_k$, and hence $\varphi(t_i) = u_j$ as required.            \end{proof} 

\begin{lemma} \label{4.3} If $\alpha$, $\beta$, and $\gamma$ are ordinals, and $m > 1$, then

(i) $\gamma$ infinite $\Rightarrow \alpha \not \equiv_4 \beta + \omega^m + \gamma^*$,
 
(ii) $\beta$ infinite $\Rightarrow \alpha \not \equiv_4 \beta + (\omega^*)^m + \gamma^*$. \end{lemma}

\begin{proof} We treat (i) and (ii) simultaneously. Let us write $\alpha = \lambda + k$ where $\lambda$ is a limit ordinal or zero and $k$ is finite. If $k = 0$, I plays the greatest point of $\gamma^*$
and wins in 2 moves. If $k > 0$, I plays the first point $a_1$ of the final $k$-block of $\alpha$. If II replies with $b_1$ in $\gamma^*$, (in (i)) or $(\omega^*)^m + \gamma^*$ (in (ii)), and $b_1$ is
not least in this set, I plays its predecessor $b_2$, and as $a_1$ has no predecessor, wins in 3 moves. If $b_1$ is the least in $\gamma^*$ (this case does not arise in (ii)), I plays the $\omega$th
point of $\gamma^*$ from the right (which may equal $b_1$). Note that this has no successor. Then if $a_2$ is II's reply, either $a_2$ is the maximum of $\alpha$, in which case I wins on the next move, or
else I can play its successor, and wins in 4 moves.

Otherwise, $b_1 \in \beta + \omega^m$ (in (i)), or $\beta$ (in (ii)). Now in case (i) I plays some $b_2 > b_1$ in $\beta + \omega^m$ with no predecessor, possible since $m > 1$. Whatever $a_2 > a_1$ I plays
in $\alpha$ has no predecessor, so I wins in 4 moves. In case (ii), player I instead plays some $b_2$ in the $(\omega^*)^m$ block with no successor, and whatever $a_2 > a_1$ I plays is either greatest in
$\alpha$ or has a successor, so I wins in 2 more moves.  \end{proof} 

\vspace{.1in}

For technical reasons, which will be explained, we focus of $\omega$-sums. Before formulating our main result on these, or indeed on $\mathbb Z$-sums, we have to deal with the fact that by Lemma \ref{2.5}, 
$\omega + \omega^* + \omega \equiv \omega$, which means that we have to make a further restriction in the monomial sums which arise in our main theorem. We summarize these in the following result.

\begin{lemma} \label{4.4} If $k, l, m, n \ge 1$, then $\omega^k + (\omega^*)^l + \omega^m \equiv \omega^n \leftrightarrow k = l = 1$ and $m = n$. More precisely, if $m \neq n$ then
$\omega^k + (\omega^*)^l + \omega^m \not \equiv_{2\min(m, n) + 2} \omega^n$, if $l > 1$ then $\omega^k + (\omega^*)^l + \omega^m \not \equiv_3 \omega^n$, and if $l = 1, m = n$, and $k > 1$, then
$\omega^k + (\omega^*)^l + \omega^m \not \equiv_6 \omega^n$.   
\end{lemma}

\begin{proof} First we note that if $m \neq n$ then player I can win by playing the least point $a_1$ of the $\omega^m$ block. Whatever point $b_1$ of $\omega^n$ player II plays, then by Lemma \ref{1.3},
player I wins on the right in at most $2\min(m, n) + 1$ more moves, since $A^{>a_1} \cong \omega^m$ and $B^{>b_1} \cong \omega^n$. So from now on assume that $m = n$. If $l > 1$ then player I plays the 
$\omega$th point $a_1$ of $(\omega^*)^l$ from the right, noting that this has no immediate successor. Whatever $b_1 \in \omega^n$ player II plays, $b_1$ has a successor $b_2$, which I now plays, and 
whatever $a_2 > a_1$ II plays in $\omega^k + (\omega^*)^l + \omega^m$, I wins (in 3 moves) by playing between $a_1$ and $a_2$. So from now on we may suppose that $l = 1$ and $m = n$.

If $k > 1$, then player I plays the last point $a_1$ in the copy of $\omega^*$. Let II's reply be $b_1 = \omega^rp_r + \ldots + \omega p_1 + p_0 \in \omega^n$. Since $A^{>a_1} \cong B^{>b_1}$ 
($\cong \omega^n$) we concentrate on the left, where $A^{<a_1} \cong \omega^k + \omega^*$ and for some coefficients $p_i < \omega$, for $i \le r$, where $r < m$, 
$B^{<b_1} \cong \omega^rp_r + \ldots + \omega p_1 + p_0$. If $p_0 = 0$ then player I wins at once since $A^{<a_1}$ has a greatest but $B^{<b_1}$ does not, so we now suppose that $p_0 \neq 0$. If 
$p_r = \ldots = p_2 = p_1 = 0$ then player I plays $a_2 = \omega$ in $A^{<a_1}$. This has no immediate predecessor, but every point of $B^{<b_1}$ does have (or is minimal), and so again I wins quickly. 
Otherwise there is a least $i > 0$ such that $p_i \neq 0$. Player I plays the first point $b_2$ of the last $\omega^i$ block in $B^{<b_1}$. Suppose that Player II plays $a_2$. If $a_2 \in \omega^*$, then 
Player I plays the least point $b_3$ of the final $p_0$ block of $(b_2, b_1)$ and wins since $b_3$ has no immediate predecessor, but any move that II can make {\em does}. Otherwise, $a_2 \in \omega^k$, so 
$(a_2, a_1) \cong \omega^k + \omega^*$ and $(b_2, b_1) \cong \omega^i + p_0$. Player I plays the least point $b_3$ of the final $p_0$ block of $(b_2, b_1)$. Once more, if II plays in $\omega^*$, I wins at
once. If II plays $a_3 \in \omega^k$, then $(a_3, a_1) \cong \omega^k + \omega^*$ and $(b_3, b_1)$ is finite, and again I wins by playing the $\omega$th point of $(a_3, a_1)$.

Finally observe that $\omega + \omega^* + \omega^m = \omega + \omega^* + \omega + \omega^m \equiv \omega + \omega^m = \omega^m$.  \end{proof} 

\vspace{.1in}

The above result explains the reason for the restriction taken on monomial sums. Let us say that a monomial sum is {\em special} if no three consecutive terms are of the form $\omega$, $\omega^*$, 
$\omega^m$ or $(\omega^*)^m$, $\omega$, $\omega^*$ for any $m \ge 1$. Note that we may reduce any monomial sum in the previous sense to an elementarily equivalent special one using Lemma \ref{4.4}. The only 
seeming exception would be if the terms eventually alternate between $\omega$ and $\omega^*$ on the right, or on the left; but then the right is $\omega + \zeta \cdot \omega \equiv \omega$ by Lemma 
\ref{2.5} (and similarly for such behaviour on the left), so we can reduce to a monomial over bounded domain (which we don't actually treat, but we could do so).

\begin{theorem} \label{4.5} If $A = \sum \{t_i: i \in \omega\}$ and $B = \sum \{u_i: i \in \omega\}$ are special monomial sums, and there are natural numbers $m_i$ and $n_i$ such that for 
each $i$, $t_i = \omega^{m_i}$ or $(\omega^*)^{m_i}$, and $u_i$ is either $\omega^{n_i}$ or $(\omega^*)^{n_i}$, then $A \cong B$ if and only if $A \equiv B$. \end{theorem}

\begin{proof} Clearly $A \cong B \Rightarrow A \equiv B$, so we concentrate on showing that $A \not \cong B \Rightarrow A \not \equiv B$. We therefore assume that $A$ and $B$ are not isomorphic, and aim to
show that they are not elementarily equivalent, that is, that player I has a winning strategy in some finite Ehrenfeucht-Fra\"iss\'e game on $A$ and $B$. For ease we suppose that whenever $m_i$ or $n_i$ 
equals 0, then the corresponding term is a power of $\omega$ rather than $\omega^*$. Since $A \not \cong B$, there is a least $N$ such that $t_N \not \cong u_N$. We shall describe an Ehrenfeucht-Fra\"iss\'e 
game on a (finite) number of moves on $A$ and $B$, whose precise value (depending on $N$) can be read off from the proof, and which is winning for player I.

\vspace{.05in}

\noindent{\bf Case 1}: $A = \alpha$ and $B = \beta$ are both ordinals. Since $A \not \cong B$, $\alpha \not \cong \beta$. Since $m_0 \ge m_1 \ge m_2 \ge \ldots$, this sequence is eventually constant, and it
follows that $\alpha < \omega^\omega$, and similarly $\beta < \omega^\omega$. Hence, by Lemma \ref{1.4}, player I has a winning strategy.

\vspace{.05in}

\noindent{\bf Case 2}: One of $A$ and $B$, $A$ say, is an ordinal, the other not.  If some member $b_1$ of $B$ has no immediate successor, then player I plays it on his first move. Whatever $a_1 \in A$ 
player II plays, $a_1$ has a successor $a_2$, which I now plays, and whatever $b_2 > b_1$ II plays in $B$, I can play $b_3$ between $b_1$ and $b_2$, and wins on the third move. If however, every member of 
$B$ has a successor, then any $i$ such that $u_i = (\omega^*)^{n_i}$ must have $n_i = 1$, and furthermore this cannot happen for consecutive terms. If such $i$ is 0, then $A$ has a least member and $B$ does 
not, so I wins in 2 moves by playing the least member of $A$ on his first move, following it by a member of $B$ less than whatever II plays. If such $i$ is non-zero, then we have 
$u_{i-1} + u_i + u_{i+1} = \omega^{n_{i-1}} + \omega^* + \omega^{n_{i+1}}$, and as every member of $B$ has a successor, $n_{i-1} > 0$, and as it is special, $n_{i-1} > 1$. On his first two moves, player I 
plays the first point $b_1$ of $u_{i-1}$ and the last point $b_2$ of $u_i$. If $a_1$ and $a_2$ are II's responses in $A$, then $(a_1, a_2)$ is well-ordered, so I wins in 4 more moves by Lemma \ref{4.3}.

\vspace{.05in}

From now on we therefore assume that neither $A$ nor $B$ is an ordinal. Let $\alpha, \beta$ be ordinals (possibly 0) such that for some $i, j$, $\bigcup_{k < i}t_k \cong \alpha$, 
$\bigcup_{k < j}u_k \cong \beta$, $t_i = (\omega^*)^{m_i}$, $u_j \cong (\omega^*)^{n_j}$.

\vspace{.05in}

\noindent{\bf Case 3}: In the notation just introduced, $\alpha \neq \beta$ (which is equivalent to $N < \max(i,j)$). By interchanging $A$ and $B$ if necessary, suppose that $\alpha > \beta$. Then by Lemma 
\ref{1.4}, $\alpha \not \equiv_{2m_0 + i} \beta$. Using the winning strategy provided by the proof of that lemma, I plays the first points $a_1, a_2, \ldots, a_i$ of $t_0, t_1, \ldots, t_{i-1}$ 
respectively, and let II's replies be $b_1, b_2, \ldots, b_i$. If the game continued as far as this, then $\beta \ge \alpha$, contrary to assumption. Hence there is a least $k$ such that 
$b_k \in \bigcup_{l \ge j}u_l$. Note that $a_k \le \min(t_{i-1})$. Now I stops playing the above sequence, and instead plays $b_{k+1} \in u_j$ so that $b_{k+1} < b_k$. Whatever $a_{k+1} < a_k$ II now plays, 
$A^{< a_{k+1}}$ is an ordinal and $B^{< b_{k+1}} \cong \beta + (\omega^*)^{n_j}$, so I can win in 4 more moves by appealing to Lemma \ref{4.3}, except in one case, which is when $u_{j-1} = \omega$ and 
$u_j = \omega^*$ (i.e. $n_j = 1$). But now, as $B$ is a special monomial sum, $u_{j+1}$ must equal $(\omega^*)^{n_{j+1}}$. If $b_{k+1}$ is the maximum of $u_j$, it therefore has no successor, so I wins in 
two more moves by playing the successor $a_{k+2}$ of $a_{k+1}$. If $b_{k+1}$ is not the maximum of $u_j$, then I plays the maximum $b_{k+2}$ of $u_j$ on his next move. As this is the first point of $B$ 
having no immediate successor , to avoid losing quickly, player II must play the first point $a_{k+2}$ of $A$, if any, having no successor (which could only be the greatest point of some $\omega^*$-block
which is followed by an $(\omega^*)^{m_l}$-block), so $a_{k+2} \in \bigcup_{l \ge i}t_l$. Now $(b_{k+1}, b_{k+2})$ is finite, and I plays the first point $a_{k+3}$ of $t_{i-1}$, which lies in 
$(a_{k+1}, a_{k+2})$ since $a_k \le \min(t_{i-1})$. This has no immediate predecessor, and therefore he wins on the next move. 

\vspace{.05in}

\noindent{\bf Case 4}: $\alpha = \beta$ (which implies that $N \ge i$), and for $i \le k \le N$, $t_k = (\omega^*)^{m_k}$ and $u_k = (\omega^*)^{n_k}$. Player I plays the greatest point $a_1$ of $t_N$. If 
II replies with $b_1 \in \alpha$, then I can win by appeal to Lemma \ref{4.3}. Note that this is inapplicable only in the case in which $t_{i-1} = \omega$ and $t_i = \omega^*$; but then, as $A$ is a
special monomial sum, $t_{i+1}$ must equal $(\omega^*)^{m_{i+1}}$, and hence the greatest point $a_2$ of $t_i$ has no successor, and I wins by playing it on his second move. If 
$b_1 \in \bigcup_{i \le k \le N}u_i$ then I restricts the play to $A^{< a_1} \cong \alpha + \gamma^*$ and $B^{< b_1} \cong \alpha + \delta^*$ for non-isomorphic ordinals $\gamma$ and 
$\delta < \omega^\omega$. (If $b_1 \in \bigcup_{i < k < N}u_i$ then $\delta < \gamma$; if $b_1 \in u_N$ then $\gamma \neq \delta$ by definition of $N$.) He plays on $\gamma^* \cup \delta^*$ using a winning 
strategy provided by Lemma \ref{1.4} until II plays in $\alpha$ (in $A$ or $B$), and then he again wins using Lemma \ref{4.3}. If $b_1 \in \bigcup_{l > N}u_l$ then I plays the greatest point $b_2$ of $u_N$, 
and the play proceeds in the way just described with the roles of $A$ and $B$ reversed.

\vspace{.05in}

\noindent{\bf Case 5}: $\alpha = \beta$, and for $i \le k \le N$, $t_k = (\omega^*)^{m_k}$ but $u_N = \omega^{n_N}$ (which implies that for $i \le k < N$, $u_k = (\omega^*)^{m_k}$). Player I again plays
the greatest point $a_1$ of $t_N$, and if II's reply $b_1$ is in $\bigcup_{l < N}u_l$, the same argument as in the previous case applies. In fact this is still true even if $b_1$ lies in the first
$\omega$ points of $u_N$, since then $B^{< b_1} \cong \bigcup_{l < N}u_l$. So we now assume that $b_1 \in \bigcup_{l \ge N}u_l$, and $B^{< b_1}$ has an initial infinite well-ordered set beyond
$\bigcup_{l < N}u_l$.

Player I now plays the least point $b_2$ of $u_N$, and by the arguments given in Case 4, player II must play some point $a_2$ of $t_{N-1}$ (since these are precisely the points for which 
$A^{<a_2} \cong B^{<b_2}$). If there is a least point $b_3$ of $B$ greater than the first $\omega$ points of $u_N$, player I plays this on his third move. Then $b_3$ has no predecessor, but whatever 
$a_3 \in (a_2, a_1)$ II now plays {\em does}, so I wins in two more moves. If there is no least point of $B$ greater than the first $\omega$ points of $u_N$, then $u_N = \omega$, and 
$u_{N+1} = (\omega^*)^{n_{N+1}}$, and as the monomial sums are special, $n_{N+1} > 1$. Player I plays some $b_3 \in u_{N+1}$ less than $b_1$. Now $(b_2, b_3) \cong \omega + (\omega^*)^{n_{N+1}}$ and
whatever $a_3$ II plays, $(a_2, a_3)$ is a reversed ordinal, so player I wins in 4 more moves by Lemma \ref{4.3} (for the reversed ordering).

\vspace{.05in}

\noindent{\bf Case 6}: $\alpha = \beta$, and for some $k \le N$, $t_l = u_l = (\omega^*)^{m_l}$ for $i \le l < k$, and $t_k = \omega^{m_k}$.

Thus $\bigcup_{l < i}t_l \cong \bigcup_{l < i}u_l \cong \alpha$, and $\bigcup_{i \le l < k}t_l = \bigcup_{i \le l < k}u_l \cong \gamma^*$, for some ordinal $\gamma$. 

Player I plays the greatest point $a_1$ of $t_{k-1}$. If II's reply lies in $\bigcup_{l < k-1}u_l$ then we may argue as in Case 4 that I can win. If $b_1 \in u_{k-1}$, but is not among the final $\omega^*$ 
points, then I plays the greatest point $b_2$ of $u_{k-1}$. Let $a_2$ be II's reply. If $(a_1, a_2)$ is well-ordered, either $b_1$ has no successor or I can play $b_3 \in (b_1, b_2)$ having no successor,
and wins in 2 more moves. Otherwise, if possible, I plays $a_3 \in (a_1, a_2)$ with no predecessor and wins in 2 more moves. If this is not possible, then $m_k = 1$ and $t_{k+1} = (\omega^*)^{m_{k+1}}$, in 
which case I plays some point $a_3$ of $t_{k+1}$ in $(a_1, a_2)$. Since the monomials are special, $m_{k+1} > 1$ and I wins by appeal to Lemma \ref{4.3}.

Next suppose that $b_1 \in \bigcup_{l \ge k}u_l$. Now player I plays the greatest point $b_2$ of $u_{k-1}$ and the whole argument is rerun with $A$ and $B$ interchanged.

The only case remaining is that $b_1$ lies in the final $\omega^*$ points of $u_{k-1}$ (or in the reversed argument, that $a_2$ lies in the final $\omega^*$ points of $t_{k-1}$). If 
$u_k = (\omega^*)^{n_k}$, player I plays the greatest point $b_2$ of $u_k$. If $a_2$ is II's reply, I plays $a_3 \in (a_1, a_2)$ if possible having no predecessor, and wins in 2 more moves. If this is not 
possible, then $m_k = 1$ and $t_{k+1} = (\omega^*)^{m_{k+1}}$, and as $A$ is a special monomial sum, $m_{k+1} > 1$. Player I plays the greatest point $a_3$ of $t_{k+1}$, and wins using Lemma \ref{4.3}.

Otherwise, $u_k = \omega^{n_k}$. Now we see that $A^{<a_1} \cong B^{< b_1} \cong \alpha + \gamma^*$, and $A^{>a_1} \cong \bigcup_{l \ge k}t_l$, and $B^{>b_1} \cong \bigcup_{l \ge k}u_l$. We may now argue 
inductively, since there are fewer blocks between $t_k$, $u_k$ and $t_N$, $u_N$ than from $t_0$, $u_0$.     \end{proof} 

\vspace{.1in}

The size of the class of such structures up to elementary equivalence is therefore $2^{\aleph_0}$. We remark that the result we would really like is just as for Theorem \ref{4.5}, but with $\omega$-sums
replaced by $\mathbb Z$-sums. The problem in proving the result in this case is as follows. Assuming that $A \not \cong B$, player I endeavours to find a finite Ehrenfeucht-Fra\"iss\'e game
on $A$ and $B$ that he can win. He plays some $a_1 \in A$ for instance, and player II replies with $b_1 \in B$. We now know that $A^{>a_1} \not \cong B^{>b_1}$ or $A^{<a_1} \not \cong B^{<b_1}$, suppose the
former. This essentially reduces the problem to an $\omega$-sum, so should be amenable to the methods presented in Theorem \ref{4.5}. The trouble is however, that the number of moves required by player I
may depend on the choice of $b_1$ by player II, and player I may not be able to predict this number in advance, which it is essential that he can do.

\vspace{.1in}

Next we consider sums of more general terms, however only in a special case, namely of two terms $A = A_1 + A_2$ and $B = B_1 + B_2$, where $A_1, A_2$ start differently (one with $\omega$, the other with 
$\omega^*$) and $B_1, B_2$ start in the same ways respectively. Generally we'd expect that $l(A, B) = \min(l(A_1, B_1), l(A_2, B_2))$, since we could imagine that the players will play either on the left, 
or on the right, but this isn't necessarily true. For an easy example consider $A_1 = \omega^*$, $B_1 = (\omega^*)^2$, $A_2 = B_2 = \omega \cdot \omega^*$, where $l(A_1, B_1) = 2$, $l(A_2, B_2) = \infty$, 
but $l(A, B) = 3$ (see below). A precise list of the ways in which the expected result fails is included in the following result. 

\begin{theorem} \label{4.6} Let $A = A_1 + A_2$ and $B = B_1 + B_2$ where $A_1 = M_0^{m_0} \cdot M_1^{m_1} \cdots M_{s-1}^{m_{s-1}}$, $A_2 = N_0^{n_0} \cdot N_1^{n_1} \cdots N_{t-1}^{n_{t-1}}$ and 
$B_1 = M_0^{p_0} \cdot M_1^{p_1} \cdots M_{u-1}^{p_{u-1}}$, $B_2 = N_0^{q_0} \cdot N_1^{q_1} \cdots N_{v-1}^{q_{v-1}}$ be sums of two monomials where 
\begin{enumerate}
\item[(i)] the $M_i$ alternate between $\omega$ and $\omega^*$, and so do the $N_i$ and,
\item[(ii)] $M_0 = \omega \Leftrightarrow N_0 = \omega^*$.
\end{enumerate}
Then $l(A, B) = \min(l(A_1, B_1), l(A_2, B_2))$, except when $M_0 = \omega^*$, in the following cases:

$s = u = 1$, $t, v \ge 2$,  and $1 = \min(m_0, p_0)$, $m_0 \neq p_0$,

$t = v = 1$, $s, u \ge 2$,  and $1 = \min(n_0, q_0)$, $n_0 \neq q_0$,

$s = 1$, $u \ge 2$, and $t \ge 2$ or $t = 1$ and $m_0 \ge 2$,

$u = 1$, $s \ge 2$, and $v \ge 2$ or $v = 1$ and $p_0 \ge 2$,

$t = 1$, $v \ge 2$, and $s \ge 2$ or $s = 1$ and $n_0 \ge 2$,

$v = 1$, $t \ge 2$, and $u \ge 2$ or $u = 1$ and $q_0 \ge 2$, \hspace{.5in} in which case $l(A,B) = 3$

$s = 1$, $u \ge 2$, and $t = m_0 = 1$,

$u = 1$, $s \ge 2$, and $v = p_0 = 1$,

$t = 1$, $v \ge 2$, and $s = n_0 = 1$,

$v = 1$, $t \ge 2$, and $u = q_0 = 1$, \hspace{.5in}
in which case $l(A, B) = 2$.           \end{theorem}

\begin{proof} First we deal with the exceptional cases. We remark that in all of these, the formula $l(A, B) = \min(l(A_1, B_1), l(A_2, B_2))$ is invalid. In the first, for instance, $l(A_1, B_1) = 2$, and 
$l(A_2, B_2) \ge 2$, which gives $\min(l(A_1, B_1), l(A_2, B_2)) = 2$, and if $s = 1$ and $u \ge 2$, then $l(A_1, B_1) = 1$.

We first consider $s = u = 1$, $A_1 = \omega^*$, $B_1 = (\omega^*)^{p_0}$, $p_0 > 1$, and $t, v \ge 2$, so that $A_2$, $B_2$ have no least. For a winning strategy in $G_4(A, B)$, player I plays the last 
point $b_1$ of the last but one copy of $\omega^*$ in $(\omega^*)^{p_0}$, so that $B_1^{> b_1} \cong \omega^*$. If II plays $a_1 \in A$ which has a successor $a_2$, then II plays $a_2$ on his second move, 
and whatever $b_2 > b_1$ in $B$ II plays, as $b_1$ has no immediate successor, I can play $b_3$ between $b_1$ and $b_2$ and win on the third move. Since every member of $A_2$ has a successor, the only other 
option is that II plays the greatest point $a_1$ of $A_1$. Now I plays the greatest point $b_2$ of $B_1$, and I must respond in $A_2$, and I wins in 2 more moves, making 4 in all (we have here used the fact 
that $B_2$ has no least). Next we have to see that II has a winning strategy in $G_3(A, B)$. For this we observe that if two linear orders each have at least 3 elements, and one has a minimum if and only if 
the other does, and similarly for maxima, then they are 2-equivalent. Note that player II can play in such a way that if the first two moves are both on the left, then for some $k \in \omega$ and $\lambda$ 
which is a limit ordinal or 0, $a_1$ is the $k$th point from the right in $A_1$ and $b_1$ is the $(\lambda + k)$th point from the right in $B_1$. The remark just made, together with the assumption that 
$A_2, B_2$ do not have least elements (since they clearly do not have greatest either) shows that $A^{< a_1} \equiv_2 B^{< b_1}$ and $A^{> a_1} \equiv_2 B^{> b_1}$ as required. Similarly, if the play takes 
place on the right, II can ensure that $a_1$ is a successor if and only if $b_1$ is a successor, which is sufficient to guarantee the truth of the same condition.

The case of $t = v = 1$, $s, u \ge 2$,  and $1 = \min(n_0, q_0)$, $n_0 \neq q_0$  is handled similarly (it is just obtained by reversing the ordering).

For the remaining `exceptional' cases, it suffices to consider that of $s = 1$ and $u \ge 2$, subdivided into $t \ge 2$, $t = 1$ and $m_0 \ge 2$, and $t = m_0 = 1$. (The others are essentially the same, 
bearing in mind that the right hand terms begin with $\omega$ instead of $\omega^*$, which is compensated for by the fact that we are now reading `in the other direction'.)

For $s = 1$, $u \ge 2$, $t \ge 2$, we have $A_1 = (\omega^*)^{m_0}$, $A_2 = \omega^{n_0} \cdot (\omega^*)^{n_1} \cdots N_{t-1}^{n_{t-1}}$. We see that I can win in 4 moves. On his first move he plays the 
greatest point of $A_1$. If II plays $b_1 \in B_2$, then I plays its successor on the next move, and wins in 3 moves. If II plays $b_1 \in B_1$, I plays $b_2 > b_1$ in $B_1$ which has no successor. Whatever 
move $a_2$ II plays, it lies in $A_2$, so has a successor, which I can play on his 3rd move, and I wins in 4 moves in all. Now we have to show that II can win the 3-move game. We remark that two infinite
linear orders are 2-equivalent if and only one has a least element if and only if the other does, and one has a greatest element if and only if the other does. Since neither $A$ nor $B$ has a greatest or 
least element, we just have to note that both orders have the same 3 options for $A^{< a}$ and $A^{> a}$ occurring, namely there are $a_i \in A$ such that $A^{< {a_1}}$ has a greatest and $A^{> {a_1}}$ has a 
least ($a_i$ is an `inner' point of a copy of $\omega$ or $\omega^*$), $A^{< {a_2}}$ has a greatest and $A^{> {a_2}}$ has no least, and $A^{< {a_3}}$ has no greatest and $A^{> {a_3}}$ has a least, and 
similarly for $B$. Thus II can respond to whatever I plays on his first move so that if $a$ and $b$ are the first moves played, then $A^{< a} \equiv_2 B^{< b}$ and $A^{> a} \equiv_2 B^{> b}$, so II can win.

Next consider $s = 1$, $u \ge 2$, $t = 1$ and $m_0 \ge 2$. Thus $A_1 = (\omega^*)^{m_0}$, $A_2 = \omega^{n_0}$. To win in 4 moves, I starts by playing the $\omega$th point $a_1$ of $A_1$ from the right. If 
II plays in $B_2$, I plays its successor and wins in 3 moves. If II plays $b_1 \in B_1$, I plays $b_2 > b_1$ in $B_1$ with no successor. Whichever $a_2 > a_1$ II plays, it has a successor, so I can win in 2 
more moves, making 4 in all. A similar argument to the previous paragraph shows that II can win in 3 moves.

If $s = 1$, $u \ge 2$, and $t = m_0 = 1$, then $A_1 = \omega^*$, $A_2 = \omega^{n_0}$, so {\em every} point of $A$ has a successor. I can win in 3 moves by playing a point $b_1$ of $B_1$ with no successor 
(the greatest point of a copy of $(\omega^*)^{p_0}$ for instance); now any response $a_1$ by II has a successor, so I can win in 2 more moves. Here we have to see that II can win in 2 moves. This is 
immediate since $A$ and $B$ have no greatest or least.

Now let use move on to the general (non-exceptional) case.

Let us write $m$ for min$(l(A_1, B_1), l(A_2, B_2))$, and suppose without loss of generality that 
$m = l(A_1, B_1) = l(M_0^{m_0} \cdot M_1^{m_1} \cdots M_{s-1}^{m_{s-1}}, M_0^{p_0} \cdot M_1^{p_1} \cdots M_{u-1}^{p_{u-1}})$ is finite. In $m$ moves, Player II has strategies for both games 
$G(M_0^{m_0} \cdot M_1^{m_1} \cdots M_{s-1}^{m_{s-1}}, M_0^{p_0} \cdot M_1^{p_1} \cdots M_{u-1}^{p_{u-1}})$ and 
$G(N_0^{n_0} \cdot N_1^{n_1} \cdots N_{t-1}^{n_{t-1}}, N_0^{q_0} \cdot N_1^{q_1} \cdots N_{v-1}^{q_{v-1}})$, and so he wins by using whichever one is required, depending on which side I's moves lie. 

A rather more complicated argument is required to show that player I can win in at most $m+1$ moves, as it isn't clear how he can ensure that the other player plays on the same side as him (which is the 
left, in view of our assumption on $m$). The best hope is that player I can `force' the play to take place on the left. If II strays to the right in the early stages, then I can ensure a quick win, and 
otherwise, his strategy will ensure that the play is to the left of already played moves, unlike in the exceptional cases. We follow the proofs of Theorems \ref{3.4} and \ref{3.5}, which compare 
monomials of the same and different lengths respectively. In the induction steps, we can ensure that this happens without too much trouble. Suppose that $A_1 = M_0^{m_0} \cdot A'$, and 
$B_1 = M_0^{m_0} \cdot B'$, and assume that I has a winning strategy $\sigma$ for the game on $A'$ and $B'$ in $l$ moves. We show that he can convert this into a winning strategy on the game between $A$ and 
$B$ in $2m_0 + l$ moves. Let I play in $A_1 \cup B_1$ according to $\sigma$ on the copies of $M_0^{m_0}$, playing the first or last point depending on whether $M_0 = \omega$ or $\omega^*$. If ever in the 
first $l$ moves, II plays in $A_2 \cup B_2$, then I can win in 2 more moves by playing the predecessor or successor of II's last move in $A_2 \cup B_2$, in the two cases, and then II must respond with a 
point {\em not} adjacent to I's last move in $A_1 \cup B_1$, and I wins by playing in between the two in $A_1 \cup B_1$. Otherwise, as $\sigma$ is winning in the game on $A'$ and $B'$, at some stage by the 
$l$th, player II plays in the same copy of $M_0^{m_0}$ more than once. By the argument given in the proof of Lemma \ref{2.9}(i) for instance, player I can win in at most $2m_0$ more moves, and this is 
because he forces the play to lie between points which have already been played {\em on the left}. 

This leaves us to deal with various `basis cases'. Some of these also follow by the same argument, since they use variants of the 2-phase lemma, and one or two require separate treatment.

There are two cases arising in the basis of Theorem \ref{3.4}(i), depending on whether $M_0 = \omega$ or $\omega^*$. In the first of these, where $A_1 = \omega^{m_0}, B_1 = \omega^{p_0}$, $m_0 < p_0$, 
player I plays the $\omega^{m_0}$th point $b_1$ of $B_1$. If II plays in $A_2$, then I wins in at most 2 more moves, and $3 \le 2m_0 + 1$. Otherwise, $a_1 \in A_1$, and I wins as usual, forcing the play to 
the left of the initial moves. 

In the second case, $A_1 = (\omega^*)^{m_0}, B_1 = (\omega^*)^{p_0}$, $m_0 < p_0$. First note that as $l(A_1, B_1) = 2m_0$ and $l(A_1, B_1) \le l(A_2, B_2)$, $A_2$ has a least if and only if $B_2$ has a 
least (and they both have no greatest). For if one of them has a least and the other does not, then $l(A_2, B_2) = 1$. If $m_0 \ge 2$, I plays the $\omega^{m_0}$th point $b_1$ of $B_1$ from the right. 
If II's move $a_1$ lies in $A_2$, then I plays its successor and wins in $\le 3$ moves. If $a_1 \in A_1$ is not the greatest in $A_1$, I plays $a_2 > a_1$ in $A_1$ so that for some $k < m_0$, 
$(a_1, a_2] \cong (\omega^*)^k$. If II replies in $B_1$, then the game continues in $A_1 \cup B_1$ on $(\omega^*)^{m_0}$ and $(\omega^*)^{m_0} \cdot 2$, so I wins in $2k+1$ more moves, $\le 2m_0 - 1$. If II 
replies in $B_2$ then I wins in 2 more moves, and $4 \le 2m_0 + 1$. If $a_1$ is the greatest member of $A_1$, I plays the $\omega^{m_0}$th point of $B_1$ from the right. Now II must reply in $A_2$, so I 
again wins in 2 more moves.

If $m_0 = 1$, we have to show that I can win in 3 moves. As this is not the exceptional case, $A_2$ or $B_2$ has a least, and by the above remark they both have a least. Hence $A_2 = \omega^{n_0}$ and
$B_2 = \omega^{q_0}$. I plays the $\omega$th point $b_1$ of $B_1$ from the right. In this case, all members of $A$ have successors, so whatever point of $A$ II plays, I can win in 2 more moves.

The basis cases for \ref{3.5}(i) are the following: $A_1 = \omega^{m_0}$, $B_1 = \omega^{p_0} \cdot (\omega^*)^{p_1} \cdots$; $A_1 = (\omega^*)^{m_0}$, $B_1 = (\omega^*)^{p_0} \cdot \omega^{p_1} \cdots$; 
$A_1 = \omega^{m_0}(\omega^*)^{m_1}$, $B_1 = \omega^{m_0}(\omega^*)^{p_1} \cdots$; $A_1 = (\omega^*)^{m_0}\omega^{m_1}$, $B_1 = (\omega^*)^{m_0} \omega^{p_1} \cdots$. 

In the first of these, $A_1$ has a least and $B_1$ does not, and the same applies to $A$ and $B$, so $l(A, B) = l(A_1, B_1) = 1$. In the third and fourth cases, since we are working on the copies of 
$\omega^{m_0}$ or $(\omega^*)^{m_0}$ an auxiliary game is used, and I can force the play to lie on the left (or else II deviates to the right and I wins quickly).

The second case is one of the exceptional ones, so has already been covered. 

The basis case for \ref{3.4}(ii) is $s = u = 2$, $A_1 = M_0^{m_0}M_1^{m_1}$, $B_1 = M_0^{p_0}M_1^{p_1}$, $m_0 < p_0$. We write $A' = M_1^{m_1}$, $B' = M_0^{p_0 - m_0}M_1^{p_1}$. Here $A$ has a greatest
element or least, and $B'$ has neither, so I wins $G(A',B')$ in 2 moves. He plays according to this winning strategy in $G(A, B)$ playing the least or greatest element of the appropriate copy of $M_0^{m_0}$ 
in $A_1 \cup B_1$, and wins in at most $2m_0 + 2$ moves (considering separately the case in which II plays in $A_2 \cup B_2$ on one of the first two moves).

The basis case for \ref{3.5}(ii) is handled similarly, which is $A_1 = M_0^{m_0}M_1^{m_1}$, $B_1 = M_0^{p_0}M_1^{p_1}M_2^{p_2} \cdots$, $m_0 \neq p_0$.

The basis cases for \ref{3.4}(iii) and \ref{3.5}(iii) are $s = u \ge 3$, $A_1 = M_0^{m_0} \cdots M_{s-1}^{m_{s-1}}$, $B_1 = M_0^{p_0} \cdots M_{u-1}^{p_{u-1}}$, $m_0 < p_0$. The same argument applies,
on $M_0^{m_0}$ blocks, taking into account dealing with a possible play by II on the right.
\end{proof}

\end{document}